\documentclass[inline, shortlabels,accepted]{uai2024} % after acceptance, for a revised version; 
\usepackage[british]{babel}

\usepackage{natbib} % has a nice set of citation styles and commands
\usepackage{mathtools} % amsmath with fixes and additions
\usepackage{booktabs} % commands to create good-looking tables
\usepackage{tikz} % nice language for creating drawings and diagrams
\usepackage{amsmath}
\usepackage{amssymb}
\usepackage{amsfonts}
\usepackage{amsthm}
\usepackage{commath}
\usetikzlibrary{positioning}
\usepackage[mathscr]{eucal}
\usepackage{dirtytalk}
\usepackage[inline]{enumitem}
\usepackage{thmtools,thm-restate,thm-autoref}
\newtheorem{theorem}{Theorem}[section]

\newtheorem{lemma}[theorem]{Lemma}

\newtheorem{example}[theorem]{Example}
\theoremstyle{definition}
\newtheorem{definition}[theorem]{Definition}

\allowdisplaybreaks

\renewcommand{\d}{\mathrm{d}}
\newcommand{\R}{\mathbb{R}}
\renewcommand{\P}{\mathbb{P}}
\newcommand{\Q}{\mathbb{Q}}
\newcommand{\K}{\mathbb{K}}
\newcommand{\E}{\mathbb{E}}
\renewcommand{\L}{\mathbb{L}}
\newcommand{\mc}[1]{\mathscr{#1}}
\newcommand{\bs}[1]{\boldsymbol{#1}}
\newcommand{\p}{\varphi}
\newcommand{\1}{\boldsymbol{1}}

\def\showauthornotes{1}
\ifnum\showauthornotes=1
\newcommand{\Authornote}[2]{{\sf\small\color{blue}{[#1: #2]}}}
\else
\newcommand{\Authornote}[2]{}
\fi

\title{Products, Abstractions and Inclusions of Causal Spaces}

\author[1]{Simon Buchholz$^\ast$}
\author[1]{Junhyung Park$^\ast$}
\author[1]{Bernhard Sch\"olkopf}
\affil[1]{%
	Empirical Inference Department\\
	Max Planck Institute for Intelligent Systems\\
	T\"ubingen, Germany
}

\begin{document}
	\maketitle
	\newcommand\nnfootnote[1]{%
		\begin{NoHyper}
			\renewcommand\thefootnote{}\footnote{#1}%
			\addtocounter{footnote}{-1}%
		\end{NoHyper}
	}
	\nnfootnote{*Equal Contribution.}
	
	\begin{abstract}
		Causal spaces have recently been introduced as a measure-theoretic framework to encode the notion of causality. While it has some advantages over established frameworks, such as structural causal models, the theory is so far only developed for single causal spaces. In many mathematical theories, not least the theory of probability spaces of which causal spaces are a direct extension, combinations of objects and maps between objects form a central part. In this paper, taking inspiration from such objects in probability theory, we propose the definitions of products of causal spaces, as well as (stochastic) transformations between causal spaces. In the context of causality, these quantities can be given direct semantic interpretations as causally independent components, abstractions and extensions. 
	\end{abstract}

	\section{Introduction}\label{sec:intro}
	Mathematical modelling of the world allows us to represent and analyse real-life situations in a quantitative manner. Depending on the aspects that one is interested in, different mathematical tools are chosen for the modelling; for example, to model how a system evolves over time, a system of differential equations can be used, and to model the randomness of events, one can use probability theory. Another aspect of the world which researchers are increasingly more interested in modelling is causality \citep{woodward2005making,russo2010causality,illari2011causality,pearl2018book}, and this is the focus of our work. 
	
	For the mathematical frameworks used in modelling, there are always ways to analyse multiple structures in a coherent manner. For example, in vector spaces, we have the notions of subspaces, product spaces and maps between vector spaces. In probability spaces, we have the notions of subspaces and restrictions, product spaces and measurable maps and transition probability kernels between spaces.
	
	In this work, we consider a modelling of the world through a recently proposed framework called \textit{causal spaces} \citep{park2023measure}. Causal spaces are a direct extension of probability spaces to encode causal information, and as such, are rigorously grounded in measure-theory. While they have some advantages over existing frameworks (e.g. structural causal models, or SCMs), such as the fact that they can easily encode cycles and continuous-time stochastic processes that are notoriously problematic in SCMs \citep{halpern2000axiomatizing,bongers2021foundations}, the theory of causal spaces is still in its infancy. In particular, \citet{park2023measure} only consider the development of \textit{single} causal spaces, and omit the discussion of construction of new causal spaces from existing ones or maps between causal spaces. The latter is of particular interest to researchers in causality for the purpose of \textit{abstraction} (see related works in Section \ref{SSrelatedworks}). When systems, humans or animals perceive the world, they consider different levels of detail depending on their ability to perceive and retain information and their level of interest. It is therefore crucial to connect the mathematical representations at varying levels of granularity in a coherent way. 
	
	In probability spaces, such notions are well-established. Product measures give rise to independent random variables, and measurable maps and probability kernels between probability spaces give rise to pushforward measures, which can be interpreted as abstractions or inclusions. Based on these concepts, and using the fact that causal spaces are a direct extension of probability spaces, we develop the notions of \textit{product causal spaces} and \textit{causal transformations}.

	\subsection{Related Works}\label{SSrelatedworks}
	The theory of causality has two dominant strands \citep{imbens2019potential}, one based on SCMs \citep{pearl2009causality,peters2017elements} and another based on potential outcomes \citep{hernan2020what,imbens2015causal}. Since concepts such as abstractions and connected components attract much more attention in the SCM community than in the potential outcomes community, we focus on comparisons with the SCM framework. 
	
	Seminal works on causal abstraction with SCMs are \citet{rubenstein2017causal} and \citet{beckers2019abstracting}, where the notions of exact transformations (to be further discussed in Section \ref{sec:comparison}), uniform transformations, abstractions, strong abstractions and constructive abstractions are proposed. \citet{beckers2020approximate} then relax these to an approximate notion. \citet{massidda2023causal} extended the notions to soft interventions, and \citet{zevcevic2023continual} to continually updated abstractions. 
	Causal feature learning is a closely related approach, that also aims to learn 
	higher level features \citep{chalupka2015visual,chalupka2016unsupervised, chalupka2017causal}
	There are also approaches based on category theory \citep{rischel2021compositional,otsuka2022equivalence,otsuka2023process} and probabilistic logic \citep{ibeling2023comparing}, all grounded in SCMs; see \citep{zennaro2022abstraction} for a review. 
	
	The notion of causal abstraction in the SCM framework has found applications in interpretations of neural networks \citep{geiger2021causal,geiger2023causal} as well as solving causal inference tasks (identification, estimation and sampling) at different levels of granularity with neural networks \citep{xia2024neural}. Moreover, \citet{zennaro2023jointly} proposed a way of \textit{learning} an abstraction from partial information about the abstraction, and demonstrates an application of causal abstraction in the SCM framework in the context of electric vehicle battery manufacturing and \citet{kekic2023targeted} learn an abstraction that explains a specific target. 
	
	\subsection{Paper Organisation}\label{subsec:paper_organisation}
	The rest of this paper is structured as follows. We first discuss the key notions from the theory of causal spaces in Section~\ref{sec:preliminaries}. Then we introduce 
	the extension of product causal spaces in Section~\ref{sec:product_spaces}
	followed by our definitions of transformations of causal spaces in Section~\ref{sec:transformations}.
	We put our definitions into context by comparing carefully to related works in Section~\ref{sec:comparison}.
	In Section~\ref{sec:results}, we then show various properties of our transformations, in particular for the subclass of abstractions. 
	
	\section{Preliminaries \& Notations}\label{sec:preliminaries}
	% We recall very briefly some essential notions from measure theoretic probability theory; for a comprehensive introduction, see, for example, \citep{cinlar2011probability,durrett2019probability}. We take a probability space as a starting point, namely, a triple \((\Omega,\mathscr{H},\mathbb{P})\) 

	We take a probability space as a starting point, namely, a triple \((\Omega,\mathscr{H},\mathbb{P})\); for a comprehensive introduction, see, for example, \citep{cinlar2011probability,durrett2019probability}.
	% satisfying the following axioms:
	% \begin{enumerate}
		%     \item \(\Omega\) is a set; 
		%     \item \(\mathscr{H}\) is a collection of subsets of \(\Omega\), called \textit{events}, such that 
		%     \begin{itemize}
			%         \item \(\Omega\in\mathscr{H}\);
			%         \item if \(A\in\mathscr{H}\), then \(\Omega\setminus A\in\mathscr{H}\);
			%         \item if \(A_1,A_2,...\in\mathscr{H}\), then \(\cup_nA_n\in\mathscr{H}\);
			%     \end{itemize}
		%     \item \(\mathbb{P}\) is a probability measure on \((\Omega,\mathscr{H}\), i.e. a function \(\mathbb{P}:\mathscr{H}\rightarrow[0,1]\) satisfying
		%     \begin{itemize}
			%         \item \(\mathbb{P}(\emptyset)=0\);
			%         \item \(\mathbb{P}(\cup_nA_n)=\sum_n\mathbb{P}(A_n)\) for any disjoint sequence \(A_n\) in \(\mathscr{H}\);
			%         \item \(\mathbb{P}(\Omega)=1\).
			%     \end{itemize}
		% \end{enumerate}
	Following \citep{park2023measure}, we additionally insist that \(\mathbb{P}\) is defined over the product measurable space \((\Omega,\mathscr{H})=\otimes_{t\in T}(E_t,\mathscr{E}_t)\) with \((E_t,\mathscr{E}_t)\) being the same standard measurable space if \(T\) is uncountable. Denote by \(\mathscr{P}(T)\) the power set of \(T\), and for \(S\in\mathscr{P}(T)\), we denote by \(\mathscr{H}_S\) the sub-\(\sigma\)-algebra of \(\mathscr{H}=\otimes_{t\in T}\mathscr{E}_t\) generated by measurable rectangles \(\times_{t\in T}A_t\), where \(A_t\in\mathscr{E}_t\) differs from \(E_t\) only for \(t\in S\) and finitely many \(t\). In particular, \(\mathscr{H}_\emptyset=\{\emptyset,\Omega\}\) is the trivial sub-\(\sigma\)-algebra of \(\Omega=\times_{t\in T}E_t\). Also, we denote by \(\Omega_S\) the subspace \(\times_{s\in S}E_s\) of \(\Omega=\times_{t\in T}E_t\), and for \(T\supseteq S\supseteq U\), we let \(\pi_{SU}\) denote the natural projection from \(\Omega_S\) onto \(\Omega_U\) and we use the shorthand \(\pi_{S}=\pi_{TS}\). We also write \(\omega_S=\pi_S \omega=(\omega_s)_{s\in S}\).
	We write \([d]=\{1,\ldots,d\}\) and \(f_\ast\P\) denotes the pushforward measure along the map \(f\), namely, \(f^*\mathbb{P}(A)=\mathbb{P}(f^{-1}(A))\).
	
	We first recall the definition of causal spaces, as given in \citet{park2023measure}. 
	\begin{definition}[Causal Spaces, {\citep[Definition 2.2]{park2023measure}}]\label{def:causal_space}
		A \textit{causal space} is defined as the quadruple \((\Omega,\mathscr{H},\mathbb{P},\mathbb{K})\), where \((\Omega,\mathscr{H},\mathbb{P})=(\times_{t\in T}E_t,\otimes_{t\in T}\mathscr{E}_t,\mathbb{P})\) is a probability space and \(\mathbb{K}=\{K_S:S\in\mathscr{P}(T)\}\), called the \textit{causal mechanism}, is a collection of transition probability kernels \(K_S\) from \((\Omega,\mathscr{H}_S)\) into \((\Omega,\mathscr{H})\), called the \textit{causal kernel on \(\mathscr{H}_S\)}, that satisfy the following axioms:
		\begin{enumerate}[(i)]
			\item\label{trivialintervention} for all \(A\in\mathscr{H}\) and \(\omega\in\Omega\), we have \(K_\emptyset(\omega,A)=\mathbb{P}(A)\);
			\item\label{interventionaldeterminism} for all \(\omega\in\Omega\), and events \(A\in\mathscr{H}_S\) and \(B\in\mathscr{H}\), we have \(K_S(\omega,A\cap B)=\1_A(\omega)K_S(\omega,B)=\delta_\omega(A)K_S(\omega,B)\).
		\end{enumerate}
	\end{definition}
	The causal kernels $K_S$ can be defined equivalently 
	as maps from $(\Omega_S,\mc{H}_S)$ to $(\Omega,\mc{H})$
	and it will be convenient to use this viewpoint occasionally in the following (this was also used implicitly in \cite{park2023measure}). This observation is explained in Appendix~\ref{app:factorization}.
	
	Next, we recall the definition of \textit{interventions}, which is \textit{the} central concept in any theory of causality.
	\begin{definition}[Interventions, {\citep[Definition 2.3]{park2023measure}}]\label{def:interventions}
		Let \((\Omega,\mathscr{H},\mathbb{P},\mathbb{K})=(\times_{t\in T}E_t,\otimes_{t\in T}\mathscr{E}_t,\mathbb{P},\mathbb{K})\) be a causal space, \(U\subseteq T\) a subset, \(\mathbb{Q}\) a probability measure on \((\Omega,\mathscr{H}_U)\) and \(\mathbb{L}=\{L_V:V\in\mathscr{P}(U)\}\) a causal mechanism on \((\Omega,\mathscr{H}_U,\mathbb{Q})\). An \textit{intervention on \(\mathscr{H}_U\) via \((\mathbb{Q},\mathbb{L})\)} is a new causal space \((\Omega,\mathscr{H},\mathbb{P}^{\text{do}(U,\mathbb{Q})},\mathbb{K}^{\text{do}(U,\mathbb{Q},\mathbb{L})})\), where the \textit{intervention measure} \(\mathbb{P}^{\text{do}(U,\mathbb{Q})}\) is a probability measure on \((\Omega,\mathscr{H})\) defined, for \(A\in\mathscr{H}\), by
		\[\mathbb{P}^{\text{do}(U,\mathbb{Q})}(A)=\int\mathbb{Q}(d\omega_U)K_U(\omega_U,A)\]
		and \(\mathbb{K}^{\text{do}(U,\mathbb{Q},\mathbb{L})}=\{K^{\text{do}(U,\mathbb{Q},\mathbb{L})}_S:S\subseteq T\}\) is the \textit{intervention causal mechanism} whose \textit{intervention causal kernels} are
		\begin{alignat*}{2}
			&K^{\text{do}(U,\mathbb{Q},\mathbb{L})}_S(\omega_S,A)\\
			&\qquad=\int L_{S\cap U}(\omega_{S\cap U},d\omega'_U)K_{S\cup U}((\omega_{S\setminus U},\omega'_U),A).
		\end{alignat*}
	\end{definition}
	We also recall the definition of \textit{causal effect}. 
	\begin{definition}[Causal Effects, {\citep[Definition B.1]{park2023measure}}]\label{def:causal_effects}
		Let \((\Omega,\mathscr{H},\mathbb{P},\mathbb{K})=(\times_{t\in T}E_t,\otimes_{t\in T}\mathscr{E}_t,\mathbb{P},\mathbb{K})\) be a causal space, \(U\subseteq T\) a subset, \(A\in\mathscr{H}\) an event and \(\mathscr{F}\) a sub-\(\sigma\)-algebra of \(\mathscr{H}\) (not necessarily of the form \(\mathscr{H}_S\) for some \(S\in\mathscr{P}(T)\)). 
		\begin{enumerate}[(i)]
			\item\label{nocausaleffect} If \(K_S(\omega,A)=K_{S\setminus U}(\omega,A)\) for all \(S\in\mathscr{P}(T)\) and all \(\omega\in\Omega\), then we say that \(\mathscr{H}_U\) has \textit{no causal effect on \(A\)}, or that \(\mathscr{H}_U\) is \textit{non-causal to \(A\)}.
			
			We say that \(\mathscr{H}_U\) has \textit{no causal effect on \(\mathscr{F}\)}, or that \(\mathscr{H}_U\) is \textit{non-causal to \(\mathscr{F}\)}, if, for all \(A\in\mathscr{F}\), the \(\sigma\)-algebra \(\mathscr{H}_U\) has no causal effect on \(A\).
			\item\label{activecausaleffect} If there exists \(\omega\in\Omega\) such that \(K_U(\omega,A)\neq\mathbb{P}(A)\), then we say that \(\mathscr{H}_U\) has an \textit{active causal effect on \(A\)}, or that \(\mathscr{H}_U\) is \textit{actively causal to \(A\)}. 
			
			We say that \(\mathscr{H}_U\) has an \textit{active causal effect on \(\mathscr{F}\)}, or that \(\mathscr{H}_U\) is \textit{actively causal to \(\mathscr{F}\)}, if \(\mathscr{H}_U\) has an active causal effect on some \(A\in\mathscr{F}\). 
			\item\label{dormantcausaleffect} Otherwise, we say that \(\mathscr{H}_U\) has a \textit{dormant causal effect on \(A\)}, or that \(\mathscr{H}_U\) is \textit{dormantly causal to \(A\)}. 
			
			We say that \(\mathscr{H}_U\) has a \textit{dormant causal effect on \(\mathscr{F}\)}, or that \(\mathscr{H}_U\) is \textit{dormantly causal to \(\mathscr{F}\)}, if \(\mathscr{H}_U\) does not have an active causal effect on any event in \(\mathscr{F}\) and there exists \(A\in\mathscr{F}\) on which \(\mathscr{H}_U\) has a dormant causal effect.
		\end{enumerate}
	\end{definition}
	Finally, we recall the definition of \textit{sources}, which allows us to connect the causal kernels to the probability measure \(\mathbb{P}\). For a sub-\(\sigma\)-algebra \(\mathscr{F}\) of \(\mathscr{H}\), we denote the \textit{conditional probability} of an event \(A\in\mathscr{H}\) given \(\mathscr{F}\) by \(\mathbb{P}_\mathscr{F}\). 
	\begin{definition}[Sources, {\citep[Definition D.1]{park2023measure}}]\label{def:sources}
		Let \((\Omega,\mathscr{H},\mathbb{P},\mathbb{K})=(\times_{t\in T}E_t,\otimes_{t\in T}\mathscr{E}_t,\mathbb{P},\mathbb{K})\) be a causal space, \(U\subseteq T\) a subset, \(A\in\mathscr{H}\) an event and \(\mathscr{F}\) a sub-\(\sigma\)-algebra of \(\mathscr{H}\). We say that \(\mathscr{H}_U\) is a \textit{(local) source} of \(A\) if \(K_U(\cdot,A)\) is a version of the conditional probability \(\mathbb{P}_{\mathscr{H}_U}(A)\). We say that \(\mathscr{H}_U\) is a \textit{(local) source} of \(\mathscr{F}\) if \(\mathscr{H}_U\) is a source of all \(A\in\mathscr{F}\). We say that \(\mathscr{H}_U\) is a \textit{global source} of the causal space if \(\mathscr{H}_U\) is a source of all \(A\in\mathscr{H}\).
	\end{definition}
	
	\section{Product Causal Spaces and Causal Independence}\label{sec:product_spaces}
	We first give the definition of the product of causal kernels, and the product of causal spaces. This constitutes the simplest way of constructing new causal spaces from existing ones. 
	\begin{definition}[Product Causal Spaces]\label{def:product}
		Suppose \(\mc{C}^1=(\Omega^1,\mc{H}^1,\P^1,\K^1)\) and \(\mc{C}^2=(\Omega^2,\mc{H}^2,\P^2,\K^2)\) with \(\Omega^1=\times_{t\in T^1}E_t\) and \(\Omega^2=\times_{t\in T^2}E_t\) are two causal spaces. For all \(S^1\subseteq T^1\) and \(S^2\subseteq T^2\), and for a pair of causal kernels \(K^1_{S^1}\in\mathbb{K}^1\) and \(K^2_{S^2}\in\mathbb{K}^2\), we define the \textit{product causal kernel} \(K^1_{S^1}\otimes K^2_{S^2}\), for \(\omega=(\omega_1,\omega_2)\in\Omega^1_{S^1}\times\Omega^2_{S^2}\) and events \(A_1\in\mathscr{H}^1\) and \(A_2\in\mathscr{H}^2\), by
		\[K^1_{S^1}\otimes K^2_{S^2}(\omega,A_1\times A_2)=K^1_{S^1}(\omega_1,A_1)K^2_{S^2}(\omega_2,A_2).\]
		This can then be extended to all of \(\mathscr{H}^1\otimes\mathscr{H}^2\) since the rectangles \(A_1\times A_2\) with \(A_1\in\mathscr{H}^1\) and \(A_2\in\mathscr{H}^2\) generate \(\mathscr{H}^1\otimes\mathscr{H}^2\). Then we define the \textit{product causal space}
		\[\mc{C}^1\otimes\mc{C}^2=(\Omega^1\times\Omega^2,\mc{H}^1\otimes\mc{H}^2,\P^1\otimes\P^2,\K^1\otimes\K^2)\]
		where the product causal mechanism \(\K^1\otimes\K^2\) is the unique family of kernels of the form \((K^1\otimes K^2)_{S^1\cup S^2}=K^1_{S^1}\otimes K^2_{S^2}\) for \(S^1\subseteq T^1\) and \(S^2\subseteq T^2\). 
	\end{definition}
	We first check that this procedure indeed produces a valid causal space. 
	\begin{restatable}[Products of Causal Spaces are Causal Spaces]{lemma}{product}\label{lem:product_space}
		The product causal space \(\mc{C}^1\otimes\mc{C}^2\) as defined in Definition \ref{def:product} is a causal space.
	\end{restatable}
	The proof of this lemma can be found in Appendix~\ref{app:proofs1}.
	
	Note that it is only for the sake of simplicity of presentation that we presented the notion of products only for two probability spaces. Indeed, we can easily extend the definition to arbitrary products of causal kernels and causal spaces, just like it is possible for products of probability spaces. 
	
	When we take a product of causal spaces, the corresponding components in the resulting causal space do not have a causal effect on each other, as the following result shows. 
	\begin{restatable}[Causal Effects in Product Spaces]{lemma}{productcausaleffect}\label{lem:product_no_causal_effect}
		Suppose \(\mc{C}^1=(\Omega^1,\mc{H}^1,\P^1,\K^1)\) and \(\mc{C}^2=(\Omega^2,\mc{H}^2,\P^2,\K^2)\) with \(\Omega^1=\times_{t\in T^1}E_t\) and \(\Omega^2=\times_{t\in T^2}E_t\) are two causal spaces. Then in \(\mc{C}^1\otimes\mc{C}^2\), 
		\begin{enumerate}[(i)]
			\item \(\mathscr{H}_{T^1}\) has no causal effect on \(\mathscr{H}_{T^2}\), and \(\mathscr{H}_{T^2}\) has no causal effect on \(\mathscr{H}_{T^1}\);
			\item \(\mathscr{H}_{T^1}\) and \(\mathscr{H}_{T^2}\) are (local) sources of each other. 
		\end{enumerate}
	\end{restatable}
	
	The proof of this Lemma is in Appendix~\ref{app:proofs1}.
	
	Product causal spaces are analogous to \textit{connected components} in graphical models -- see, for example, \citep{sadeghi2023axiomatization}.
	
	\subsection{Causal Independence}\label{subsec:causal_independence}
	Recall that, in probability spaces, two events \(A\) and \(B\) are \textit{independent} with respect to the measure \(\mathbb{P}\) if \(\mathbb{P}(A\cap B)=\mathbb{P}(A)\mathbb{P}(B)\), i.e. the probability measure is the product measure.  Moreover, two \(\sigma\)-algebras are independent if each pair of events from the two \(\sigma\)-algebras are independent\footnote{Many authors take the view that the notion of independence is truly where probability theory starts, as a distinct theory from measure theory \citep[p.82, Section II.5]{cinlar2011probability}.}. Similarly, for a sub-\(\sigma\)-algebra \(\mathscr{F}\) of \(\mathscr{H}\), two events \(A\) and \(B\) are \textit{conditionally independent given \(\mathscr{F}\)} if \(\mathbb{P}_\mathscr{F}(A\cap B)=\mathbb{P}_\mathscr{F}(A)\mathbb{P}_\mathscr{F}(B)\) almost surely, and two \(\sigma\)-algebras are conditionally independent given \(\mathscr{F}\) if each pair of events from the two \(\sigma\)-algebras are conditionally independent given \(\mathscr{F}\). 
	
	Now we give a definition of causal independence that is analogous to conditional independence. 
	\begin{definition}[Causal Independence]\label{def:causal_independence}
		Consider a causal space \(\mc{C}=(\Omega=\times_{t\in T}E_t,\mc{H}=\otimes_{t\in T}\mathscr{E}_t,\mathbb{P},\mathbb{K})\). Then for \(U\subseteq T\), two events \(A,B\in\mathscr{H}\) are \textit{causally independent on \(\mathscr{H}_U\)} if, for all \(\omega\in\Omega\), 
		\[K_U(\omega,A\cap B)=K_U(\omega,A)K_U(\omega,B).\]
		We say that two sub-\(\sigma\)-algebras \(\mathscr{F}_1\) and \(\mathscr{F}_2\) are \textit{causally independent on \(\mathscr{H}_U\)} if each pair of events from \(\mathscr{F}_1\) and \(\mathscr{F}_2\) are causally independent on \(\mathscr{H}_U\). 
		% Then for \(T_1,T_2\subseteq T\) with \(T_1\cap T_2=\emptyset\), the corresponding \(\sigma\)-algebras \(\mathscr{H}_{T_1}\) and \(\mathscr{H}_{T_2}\) are \textit{causally independent} if, for all subsets \(S_1\subseteq T_1\) and \(S_2\subseteq T_2\), and all \(A\in\mathscr{H}_{T_1}\otimes\mathscr{H}_{T_2}\) and \(\omega\in\Omega\), we have
		% \[K_{S_1\cup S_2}(\omega,A)=K_{S_1}\otimes K_{S_2}(\omega,A).\]
	\end{definition}
	Semantically, causal independence should be interpreted as follows: if \(A\) and \(B\) are causally independent on \(\mathscr{H}_U\), then they are independent once an intervention has been carried out on \(\mathscr{H}_U\). Note also that causal independence is really about the causal kernels, and has nothing to do with the probability measure \(\mathbb{P}\) of the causal space. Indeed, it is possible for \(A\) and \(B\) to be causally independent but not probabilistically independent, or causally independent but not conditionally independent, or vice versa. Let us illustrate with the following simple examples. 
	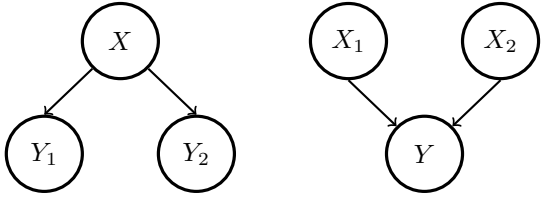
\begin{figure}[t]
		\begin{center}
			\begin{tikzpicture}[roundnode/.style={circle, draw=black, very thick, minimum size=10mm}]
				\node[roundnode, align=center] (y1) at (-1, 0) {\(Y_1\)};
				\node[roundnode, align=center] (y2) at (1, 0) {\(Y_2\)};
				\node[roundnode, align=center] (x) at (0, 1.5) {\(X\)};
				\draw[thick, ->](x.south west) to (y1.north);
				\draw[thick, ->](x.south east) to (y2.north);
				\node[roundnode, align=center] (y) at (4, 0) {\(Y\)};
				\node[roundnode, align=center] (x1) at (3, 1.5) {\(X_1\)};
				\node[roundnode, align=center] (x2) at (5, 1.5) {\(X_2\)};
				\draw[thick, ->](x1.south) to (y.north west);
				\draw[thick, ->](x2.south) to (y.north east);
			\end{tikzpicture}
		\end{center}
		\caption{Graphs of SCMs in Example \ref{ex:causal_independence}.}
		\label{fig:independence}
	\end{figure}
	\begin{example}[Causal Independence]\label{ex:causal_independence}
		\begin{enumerate}[(i)]
			\item Consider three variables \(X\), \(Y_1\) and \(Y_2\) related through the equations
			\[X=N,\quad Y_1=X+U_1,\quad Y_2=X+U_2,\]
			where \(N\), \(U_1\) and \(U_2\) are standard normal variables (see Figure~\ref{fig:independence} left). We denote by \(\P\) their joint distribution on \(\R^3\), and we identify this SCM with the causal space \((\mathbb{R}^3,\mathscr{B}(\mathbb{R}^3),\mathbb{P},\mathbb{K})\)\footnote{Here, \(\mc{B}\) represents the Borel \(\sigma\)-algebra.}, where \(\K\) is obtained via the above structural equations. Then it is clear to see that \(Y_1\) and \(Y_2\) are causally independent on \(\mathscr{H}_X\), since, for every \(x\), and \(A,B\in\mc{B}(\mathbb{R})\), \(K_X(x,\{Y_1\in A,Y_2\in B\})\) is bivariate-normally distributed with mean \((x,x)\) and identity covariance matrix, and so
			\begin{alignat*}{2}
				K_X&(x,\{Y_1\in A,Y_2\in B\})\\
				&=K_X(x,\{Y_1\in A\})K_X(x,\{Y_2\in B\}).
			\end{alignat*}
			By the same reasoning, \(Y_1\) and \(Y_2\) are conditionally independent given \(\mathscr{H}_X\). However, it is clear that they are unconditionally dependent, because they both depend on the value of \(X\). 
			\item Now consider three variables \(X_1\), \(X_2\) and \(Y\) related through the equations
			\[X_1=N_1,\quad X_2=N_2,\quad Y=X_1+X_2+U\]
			where \(N_1\), \(N_2\) and \(U\) are standard normal variables (see Figure \ref{fig:independence} right). We denote by \(\P\) their joint distribution on \(\R^3\), and we identify this SCM with the causal space \((\mathbb{R}^3,\mathscr{B}(\mathbb{R}^3),\mathbb{P},\mathbb{K})\), where \(\K\) is obtained via the above structural equations. Then it is clear that \(X_1\) and \(X_2\) are probabilistically independent. They are also causally independent on \(\mathscr{H}_Y\), since, for any \(A,B\in\mc{B}(\R)\),
			\begin{alignat*}{2}
				K_Y(y,\{X_1\in A,X_2\in B&\})=\P(X_1\in A,X_2\in B)\\
				&=\P(X_1\in A)\P(X_2\in B).
			\end{alignat*}
			However, it is clear that they are conditionally dependent given \(\mathscr{H}_Y\). 
		\end{enumerate}
	\end{example}
	Again, causal independence is only defined for two \(\sigma\)-algebras for the sake of notational convenience; it can easily be extended to arbitrary collections of \(\sigma\)-algebra. 
	
	\section{Transformations of Causal Spaces}\label{sec:transformations}
	Consider causal spaces \(\mc{C}^1=(\Omega^1,\mc{H}^1,\mathbb{P}^1,\mathbb{K}^1)\) and \(\mc{C}^2=(\Omega^2,\mc{H}^2,\mathbb{P}^2,\mathbb{K}^2)\) with \(\Omega^1=\times_{t\in T^1}E_t\) and \(\Omega^2=\times_{t\in T^2}E_t\). We want to define transformations between causal spaces \(\mc{C}^1\) and \(\mc{C}^2\). These transformations shall, on the one hand, preserve aspects of the causal structure, i.e., the spaces \(\mc{C}^1\) and \(\mc{C}^2\) shall still describe essentially the same system. On the other hand, they shall be flexible so that different types of mappings between causal spaces can be captured. 
	
	We focus on transformations that preserve individual variables or combine them in a meaningful way. This relation will be encoded by a map \(\rho:T^1\to T^2\), which can be interpreted as encoding the fact that \(S\subseteq T^2\) depends only on the variables indexed by \(\rho^{-1}(S)\). Deterministic maps are not sufficiently expressive for our purposes and we therefore focus on stochastic maps, i.e., on 
	probability kernels from measurable spaces \((\Omega^1,\mathscr{H}^1)\) to 
	\((\Omega^2,\mathscr{H}^2)\). 
	\begin{definition}[Admissible Maps]\label{def:admissible}
		Suppose that \(\kappa:\Omega^1\times\mc{H}^2\to [0,1]\) is a probability kernel and \(\rho:T^1\to T^2\) is a map. Then we call the pair \((\kappa,\rho)\) \textit{admissible} if \(\kappa(\cdot,A)\) is \(\mc{H}^1_{\rho^{-1}(S)}\) measurable for all \(S\subset \rho(T^1)\)
		and \(A\in \mc{H}^2_{S}\).
	\end{definition}
	One difference between probability theory and causality seems to be that the latter requires the notion of variables (equivalently a product structure of the underlying space) that define entities that can be intervened upon. For a meaningful relation between two causal spaces, their interventions should be related, which requires some preservation of variables. The definition of admissible maps captures the fact that variables from \(\rho^{-1}(S)\) are combined to form a new summary collection of variables indexed by \(S\).
	% Note that the definition is given for individual \(t\in\rho(T^1)\), but it can be shown to be true for any \(S\subseteq\rho(T^1)\subseteq T^2\).
	% \begin{lemma}
		%     For any \(S\subseteq\rho(T^1)\subseteq T^2\) and any \(A\in\mathscr{H}^2_S\), the function \(\kappa(\cdot,A)\) is measurable with respect to \(\mathscr{H}^1_{\rho^{-1}(S)}\). 
		% \end{lemma}
	% \begin{proof}
		%     First, see that, for any \(t\in S\) and any \(A\in\mathscr{H}^2_{\{t\}}\), \(\kappa(\cdot,A)\) is measurable with respect to \(\mathscr{H}^1_{\rho^{-1}(t)}\) by definition. Then note that for any finite subset \(S'\subseteq S\), 
		% \end{proof}
	% Note that this condition  implies that for $S\subseteq \rho(T^1)\subseteq T^2$ and $A\in \mc{H}^2_{S}$
	% the function $\kappa(\cdot, A)$ is measurable with respect to $\mc{H}^1_{\rho^{-1}(S)}$. Indeed, by definition this is true
	% for any $A\in \mc{H}^2_{\{t\}}$ for any $t\in S$ because $\mc{H}^1_{\rho^{-1}(t)}\subset \mc{H}^1_{\rho^{-1}(S)}$.
	% So this is true for all $A\in \mc{H}^2_{S}$. \snote{Any more details, is this true? What about infinite products?}\jnote{Why don't we take this as the definition, instead of individual \(t\)'s? Not sure if I can show it. }
	
	We now require maps between causal spaces to respect the distributional and interventional structure in the following sense.
	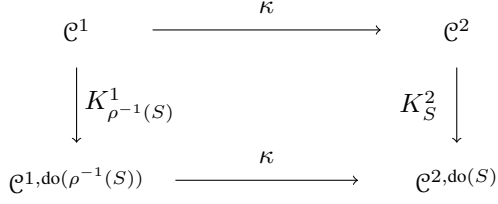
\begin{figure}[t]
		\centering
		\begin{tikzpicture}
			\node at (0, 0) {\(\mathscr{C}^1\)};
			\node at (5,0) {\(\mathscr{C}^2\)};
			\node at (0, -2) {\(\mathscr{C}^{1,\text{do}(\rho^{-1}(S))}\)};
			\node at (5, -2) {\(\mathscr{C}^{2,\text{do}(S)}\)};
			\draw[->] (1,0) -- (4,0);
			\node at (2.5, 0.3) {\(\kappa\)};
			\draw[->] (0,-0.5) -- (0, -1.5);
			\draw[->] (1.3, -2) -- (3.7, -2);
			\draw[->] (5, -0.5) -- (5, -1.5);
			\node at (2.5, -1.7) {\(\kappa\)};
			\node at (0.7,-1) {\(K^1_{\rho^{-1}(S)}\)};
			\node at (4.5, -1) {\(K^2_S\)};
		\end{tikzpicture}
		\caption{Interventional Consistency Definition \ref{def:causal_transformation} Equation (\ref{eq:interventional_consistency}) -- intervention and transformation commute.}
		\label{fig:interventional_consistency}
	\end{figure}
	\begin{definition}[Causal Transformations]\label{def:causal_transformation}
		A \textit{transformation} of causal spaces, or a \textit{causal transformation}, \(\p:\mc{C}^1\to\mc{C}^2\) is an admissible pair \(\p=(\kappa,\rho)\) satisfying the following two properties. 
		\begin{enumerate}[(i)]
			\item The map satisfies  \textit{distributional consistency}, i.e., for \(A\in\mc{H}^2\)
			\begin{equation}\label{eq:distributional_consistency}
				\int\P^1(\d\omega)\,\kappa(\omega,A)=\P^2(A).
			\end{equation}
			\item The map satisfies \textit{interventional consistency}, i.e., for all \(A\in\mc{H}^2_{\rho(T^1)}\), \(S\subset\rho(T^1)\), and \(\omega\in\Omega^1\) the following holds 
			\begin{equation}\label{eq:interventional_consistency}
				\begin{split}
					\int K^1_{\rho^{-1}(S)}(\omega,\d\omega')&\kappa(\omega',A)\\
					&=\int\kappa(\omega,\d\omega')K^2_{S}(\omega',A).
				\end{split}
			\end{equation}
		\end{enumerate}
	\end{definition}
	Interventional consistency requires that interventions and causal transformations commute, i.e., the result of first intervening and then applying the transformation is the same as intervening on the target after the transformation -- see Figure~\ref{fig:interventional_consistency}. We emphasise that in Definition~\ref{def:admissible} and \ref{def:causal_transformation} we do not prescribe conditions for added components indexed by \(T^2\setminus\rho(T^1)\). Further, we remark that as a special case, we can accommodate deterministic maps \(f:\Omega_1\to\Omega_2\) by considering the associated probability kernel \(\kappa_f(\omega,A)=\bs{1}_A(f(\omega))\). In this case, the admissibility condition reduces to the statement that \(\pi_{S}\circ f\) is measurable with respect to \(\mc{H}^1_{\rho^{-1}(S)}\) for all \(S\subset\rho(T^1)\) and distributional consistency 
	becomes, for \(A\in\mc{H}^2\),
	\begin{alignat*}{2}
		\P^2(A)&=\int\P^1(\d\omega)\kappa(\omega,A)\\
		&=\int\P^1(\d\omega)\bs{1}_A(f(\omega))\\
		&=\P^1(f^{-1}(A))
	\end{alignat*}
	so \(f_\ast\P^1=\P^2\) is the pushforward measure of \(\P^1\) along \(f\). Interventional consistency then reads
	\begin{equation}\label{eq:int_consistent_f}
		K^1_{\rho^{-1}(S)}(\omega,f^{-1}(A))=K^2_{S}(f(\omega),A)
	\end{equation}
	for all \(A\in\mc{H}^2_{\rho(T^1)}\), \(S\subset\rho(T^1)\), and \(\omega\in\Omega^1\). Alternatively this can be expressed as
	\[f_\ast K^1_{\rho^{-1}(S)}(\omega,A)=K^2_{S}(f(\omega),A).\]
	where the push-forward acts on the measure defined by the probability kernel for some fixed \(\omega\). Henceforth, with a slight abuse of notation, we denote deterministic maps by \((f,\rho)\) without resorting to the associated probability kernel.
	
	\subsection{Examples}\label{subsec:examples}
	Let us provide four prototypical examples of maps between causal spaces that are covered by this definition.
	Here we resort to the language of SCMs because they are a convenient framework that fits into causal spaces.
	\begin{figure}[t]
		\begin{center}
			\begin{tikzpicture}[roundnode/.style={circle, draw=black, very thick, minimum size=10mm}]
				\node[roundnode, align=center] (y1) at (-1, 0) {\(Y_1\)};
				\node[roundnode, align=center] (y2) at (1, 0) {\(Y_2\)};
				\node[roundnode, align=center] (x1) at (-1, 1.5) {\(X_1\)};
				\node[roundnode, align=center] (x2) at (1, 1.5) {\(X_2\)};
				\draw[thick, ->](x1.south) to (y1.north);
				\draw[thick, ->](x1.south east) to (y2.north west);
				\draw[thick, ->](x2.south) to (y2.north);
				\draw[thick, ->](x2.south west) to (y1.north east);
				\node[roundnode, align=center] (x) at (5, 1.5) {\(X\)};
				\node[roundnode, align=center] (y) at (5, 0) {\(Y\)};
				\draw[thick, ->](x.south) to (y.north);
				\draw[thick, ->] (2,0.75) -- (4, 0.75);
				\node at (3, 1.2) {\(X=X_1+X_2\)};
				\node at (3, 0.3) {\(Y=Y_1+2Y_2\)};
			\end{tikzpicture}
		\end{center}
		\caption{Abstraction of SCMs in Example~\ref{ex:abs}.}
		\label{fig:abstraction}
	\end{figure}
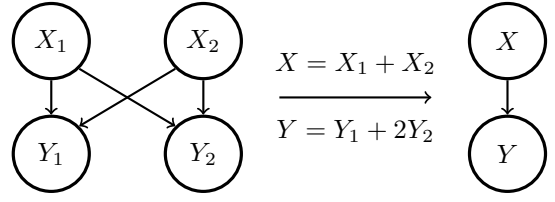
	\begin{example}[Abstraction]\label{ex:abs}
		We consider four variables \(X_1\), \(X_2\), \(Y_1\), and \(Y_2\) which are related through the equations
		\begin{alignat*}{3}
			X_1&=N_1,&X_2&=N_2,\\
			Y_1&=3X_1+X_2+U_1,\qquad&Y_2&=X_2+U_2
		\end{alignat*}
		where \(U_1\), \(U_2\), \(N_1\), \(N_2\) are independent standard normal variables. We denote by \(\P\) their joint distribution on \(\R^4\). Consider 
		\[X=N,\qquad Y=3X+U\]
		where \(N\sim N(0,2)\) and \(U\sim N(0,5)\). Denote their joint distribution on \(\R^2\) by \(\Q\). We identify the two SCMs with causal spaces \((\R^4, \mc{B}(\R^4),\P,\K)\) and \((\R^2,\mc{B}(\R^2),\Q,\L)\) as explained in \citep[Section 3.1]{park2023measure}.
		
		Consider the deterministic map \(f:\R^4\to\R^2\) given by \(f(x_1,x_2,y_1,y_2)=(x_1+x_2,y_1+2y_2)\) and the map \(\rho:[4]\to[2]\) given by \(\rho(1)=\rho(2)=1\), \(\rho(3)=\rho(4)=2\). Clearly, the pair \((f,\rho)\) is admissible as defined in Definition~\ref{def:admissible}. It can be checked that 
		% \begin{align}
			% \begin{split}
				%     \E((X_1+X_2)^2)=\E(N_1^2+N_2^2)=2=\E(X^2)\\
				%     \E((Y_1+Y_2)^2)= \E((2N_1)^2+N_2^2+U_1^2+U_2^2)
				%     =7 = \E( (\frac{3}{2}N)^2 + U^2 )=\E(Y^2)
				%     \\
				%     \E (X_1+X_2)(Y_1+Y_2) = 
				%     \E 2N_1^2+N_2^2=3 = \E(\frac{3}{2}N^2)=\E(XY)
				% \end{split}
			% \end{align}
		\begin{alignat*}{2}
			\E[(X_1+X_2)^2]&=2=\E[X^2]\\
			\E[(Y_1+2Y_2)^2]&=23=\E[Y^2]\\
			\E[(X_1+X_2)(Y_1+2Y_2)]&=6=\E[XY]
		\end{alignat*}
		which implies that \(f_\ast\P=\Q\) because both distributions are centered Gaussian and their covariance matrices agree.
		
		The non-trivial causal consistency relation \eqref{eq:interventional_consistency} concerns interventions on \(\{X_1,X_2\}\) and \(X\) and on \(\{Y_1,Y_2\}\) and \(Y\). Note that 
		\begin{alignat*}{2}
			K_{\{X_1,X_2\}}&((x_1,x_2,y_1,y_2),\cdot)\\
			&= \delta_{(x_1,x_2)}\otimes N\left(\begin{pmatrix}3x_1+x_2\\x_2\end{pmatrix},\mathrm{Id}_2\right).
		\end{alignat*}
		Then we obtain
		\begin{alignat*}{2}
			f_\ast  K_{\{X_1,X_2\}}&((x_1,x_2,y_1,y_2),\cdot)\\
			&=\delta_{x_1+x_2}\otimes N(3x_1+3x_2,5).
		\end{alignat*}
		On the other hand, we find
		\begin{alignat*}{2}
			L_{X}((x,y),\cdot )&=\delta_x\otimes N\left(3 x, 5\right)\\
			\Rightarrow  L_{X}(f(x_1,x_2,y_1,y_2),\cdot)&=\delta_{x_1+x_2}\otimes N(3x_1+3x_2,5)
		\end{alignat*}
		so that we see that \eqref{eq:int_consistent_f} holds in this case. Similarly, we obtain
		\begin{alignat*}{2}
			K_{\{Y_1,Y_2\}}((x_1,x_2,y_1,y_2),\cdot)&=N(0,\mathrm{Id}_1)\otimes\delta_{(y_1,y_2)},\\
			L_{Y}((x,y),\cdot)&=N(0,2)\otimes\delta_y.
		\end{alignat*}
		We again find 
		\begin{alignat*}{2}
			f_\ast K_{\{Y_1,Y_2\}}&((x_1,x_2,y_1,y_2),\cdot)
			\\
			&=L_{Y}((x_1+x_2,y_1+2y_2),\cdot).
		\end{alignat*}
	\end{example}
	This example shows abstraction, i.e., we obtain a transformation to a more coarse-grained view of the system. 
	Note that interventional consistency is quite restrictive to satisfy, e.g., here it is crucial that
	all distributions are Gaussian so that all conditional distributions are also Gaussian. 
	
	Next, we consider an example that allows us to embed a causal space in a larger space that adds an independent disjoint system. For this, we make use of the definition of product causal spaces (Definition \ref{def:product}). In this case, the transformation is stochastic. 
	\begin{figure}[t]
		\begin{center}
			\begin{tikzpicture}[roundnode/.style={circle, draw=black, very thick, minimum size=10mm}]
				\node at (0, 0) {\(\mc{C}^1\)};
				\node at (4.5, 0) {\(\mc{C}^1\otimes\mc{C}^2\)};
				\draw[thick, ->] (1,0) -- (3, 0);
				\node at (2, 0.4) {\(\rho(t)=t\)};
				\node at (2, -0.4) {\(\kappa(\omega,\cdot)=\delta_\omega\otimes\mathbb{P}^2\)};
			\end{tikzpicture}
		\end{center}
		\caption{Inclusions of component causal spaces into the product (Example~\ref{ex:incl}).}
		\label{fig:inclusion}
	\end{figure}
	
	\begin{example}[Inclusion]\label{ex:incl}
		Let \(\mc{C}^1=(\Omega^1,\mc{H}^1,\P^1,\K^1)\) and \(\mc{C}^2=(\Omega^2,\mc{H}^2,\P^2,\K^2)\) be two causal spaces, with \(\Omega^1=\times_{t\in T^1}E_t\) and \(\Omega^2=\times_{t\in T^2}E_t\). We define an inclusion map \((\kappa,\rho):\mc{C}^1\to\mc{C}^1\otimes\mc{C}^2\) by considering \(\rho(t)=t\) for \(t\in T^1\) and \(\kappa(\omega,\cdot)=\delta_{\omega}\otimes\P^2\) (see Figure~\ref{fig:inclusion}). This pair is clearly admissible and satisfies distributional consistency:
		\begin{alignat*}{2}
			\int\P^1(\d\omega)\kappa(\omega,A_1\times A_2)&=\int\P^1(\d\omega)\bs{1}_{A_1}(\omega)\P^2(A_2)\\
			&=\P^1(A_1)\P^2(A_2).
		\end{alignat*}
		Moreover, for any \(S\subset T^1\), \(\omega\in\Omega^1\), \(A_1\in\mathscr{H}^1\) and \(A_2\in\mathscr{H}^2\), we have
		\[\int K_S^1(\omega,\d\omega')\kappa(\omega',A_1\times A_2)=\P^2(A_2)K_S^1(\omega,A_1)\]
		and also, 
		\begin{alignat*}{2}
			\int&\kappa(\omega,d\omega'_1d\omega'_2)K^1_S\otimes K^2_\emptyset((\omega'_1,\omega'_2),A_1\times A_2)\\
			&=\int\kappa(\omega,d\omega'_1d\omega'_2)K^1_S(\omega'_1,A_1)K^2_\emptyset(\omega'_2,A_2)\\
			&=K^1_S(\omega,A_1)\int\P^2(d\omega'_2)\P^2(A_2)\\
			&=\mathbb{P}^2(A_2)K^1_S(\omega,A_1).
		\end{alignat*}
		where we used the condition on \(K_\emptyset\) in Definition~\ref{def:causal_space}. By the usual monotone convergence theorem arguments, we have that, for any \(A\in\mathscr{H}^1\otimes\mathscr{H}^2\), 
		\begin{alignat*}{2}
			\int&K^1_S(\omega,d\omega')\kappa(\omega',A)\\
			&=\int\kappa(\omega,d\omega'_1d\omega'_2)K^1_S\otimes K^2_\emptyset((\omega'_1,\omega'_2),A_1\times A_2).
		\end{alignat*}
		Thus, interventional consistency holds, in this case even for all sets \(A\), not just for those measurable with respect to \(\mc{H}^2_{\rho(T^1)}\).
	\end{example}
	This shows that we can consider causal maps including our system into a larger system containing additional independent components.
	
	Finally, we consider a more involved embedding example.
	
	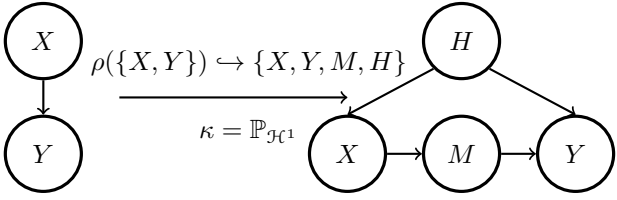
\begin{figure}[t]
		\begin{center}
			\begin{tikzpicture}[roundnode/.style={circle, draw=black, very thick, minimum size=10mm}]
				\node[roundnode, align=center] (y1) at (-1, 0) {\(Y\)};
				\node[roundnode, align=center] (x1) at (-1, 1.5) {\(X\)};
				\draw[thick, ->](x1.south) to (y1.north);
				\draw[thick, ->] (0,0.75) -- (3, 0.75);
				\node[roundnode, align=center] (x2) at (3, 0) {\(X\)};
				\node[roundnode, align=center] (m) at (4.5, 0) {\(M\)};
				\node[roundnode, align=center] (y2) at (6, 0) {\(Y\)};
				\node[roundnode, align=center] (h) at (4.5, 1.5) {\(H\)};
				\draw[thick, ->](x2.east) to (m.west);
				\draw[thick, ->](m.east) to (y2.west);
				\draw[thick, ->](h.south east) to (y2.north);
				\draw[thick, ->](h.south west) to (x2.north);
				\node at (1.7, 1.2) {\(\rho(\{X,Y\})\hookrightarrow\{X,Y,M,H\}\)};
				\node at (1.7, 0.3) {\(\kappa=\P_{\mc{H}^1}\)};
			\end{tikzpicture}
		\end{center}
		\caption{Inclusions of SCMs (Example~\ref{ex:sub}).}
		\label{fig:sub}
	\end{figure}
	
	\begin{example}[Inclusion of SCMs]\label{ex:sub}
		Consider the following SCM 
		\begin{alignat*}{3}
			H&=N_H,&X&=H+N_X,\\
			M&=X+N_M,\qquad&Y&=M+H+N_Y.
		\end{alignat*}
		We denote the joint distribution of \((X,Y,M,H)\) by \(\mathbb{P}\), and the marginal distribution on \((X,Y)\) by \(\mathbb{P}^{XY}\). 
		
		We consider a causal space \(\mc{C}^1=(\Omega^1,\mc{H}^1,\P^{XY},\K)\) that represents the pair \((X,Y)\), where \(\Omega^1=\mathbb{R}^2\) and \(\mc{H}^1=\mathscr{B}(\R^2)\), and a causal space \(\mc{C}^2=(\Omega^2,\mc{H}^2,\P,\L)\) representing the full SCM, where \(\Omega^2=\mathbb{R}^4\) and \(\mathscr{H}^2=\mathscr{B}(\mathbb{R}^4)\), i.e., it contains in addition a mediator and a confounder. The causal mechanisms \(\K\) and \(\L\) are derived from the SCM. Then we consider the obvious \(\rho\) that embeds \(\{X,Y\}\) into \(\{X,Y,M,H\}\) and, for \(A\in\mathscr{H}^2\), 
		\[\kappa(\cdot,A)=\P_{\mathscr{H}^1}(A).\]
		Clearly, this pair is admissible because on the variables \(X\) and \(Y\)  we use the identity transformation. Distributional consistency follows by
		\begin{alignat*}{2}
			\int\kappa((x,y),A)\P^{XY}(d(x,y))&=\int\P_{\mathscr{H}^1}(A)d\P^{XY}\\
			&=\P(A).
		\end{alignat*}
		Interventional consistency also holds so that \((\kappa,\rho)\) is indeed  a causal transformation. For a proof of this fact we refer to the  more general result in Lemma~\ref{lem:scm}.

	\end{example}
	This example therefore shows that we can embed a system in a larger system that captures a more accurate description.
	% Let us give one last example of a causal transformation that we will use
	% to show that for general stochastic maps $\kappa$ there are not necessarily close connections between domain and target, e.g., interventional consistency
	% does not restrict the causal kernels on the target much.
	% \begin{example}\label{ex:trivial}
		%     Consider two causal spaces $\mc{C}^1$ and $\mc{C}^2$.  Then the pair $(\kappa, \rho)$
		%     where $\rho$ is arbitrary and $\kappa(\omega, \cdot)=\P^2(\cdot)$ for all $\omega\in \Omega^1$
		%     defines a causal map if for all $S^2\subset T^2$
		%     \begin{align}
			%         \int \P^2(\d \omega) K_{S^2}(\omega, A)=\P^2(A).
			%     \end{align}
		% \end{example}
	% We find that $\kappa$ does not depend on $\omega$ so that $(\kappa,\rho)$ is admissible.
	% Clearly, for every probability measure $\P^1$
	% \begin{align}
		%     \P^2(A)=\int \P^1(\d\omega) \kappa(\omega, A).
		% \end{align}
	% For interventional consistency we find the necessary and sufficient condition
	% \begin{align}
		% \begin{split}
			%     \int K_{\rho^{-1}(S^2)}^1&(\omega, \d \omega')\kappa(\omega', A)
			%     = \P^2(A)
			%     \\&\overset{!}{=} \int \P^2(\d \omega) K_{S^2}^2(\omega, A)
			%     \end{split}
		% \end{align}
	% In general, this condition is consistent with many different kernels $K^2_{S^2}$.
	
	\subsection{Abstractions}\label{subsec:abstractions}
	Note that Example~\ref{ex:abs} is different from Examples~\ref{ex:incl} and \ref{ex:sub} in that it compresses the representation
	while the other two consider an extension of the system. 
	As these are different objectives, we consider the following definition.
	\begin{definition}[Abstractions]\label{def:abstraction}
		The pair of maps \((\kappa,\rho)\) between measurable spaces \((\Omega^1,\mc{H}^1)=\otimes_{t\in T^1}(E_t,\mathscr{E}_t)\) and \((\Omega^2,\mc{H}^2)=\otimes_{t\in T^2}(E_t,\mathscr{E}_t)\) is called an \textit{abstraction} if \(\rho:T^1\to T^2\) is surjective.
	\end{definition}
	In the case of abstractions it is often sufficient to consider deterministic maps, motivating the following definition.
	\begin{definition}[Perfect Abstractions]\label{def:perfect_abstraction}
		An abstraction \((\kappa,\rho)\) is called a perfect abstraction if \(\kappa\) is deterministic, i.e., \(\kappa=\kappa_f\) for some measurable \(f:\Omega^1\to\Omega^2\), and moreover \(f\) is surjective.
	\end{definition}
	We finally remark that one further setting of potential interest would be to consider the inverse of an abstraction, i.e., a setting where a summary variable \(X\) is mapped to a more detailed description \((X_1,X_2)\). However, to accommodate such transformations we need a slightly different framework than the one presented here. Roughly, we need to consider \(\rho:T^1\to\mc{P}(T^2)\) with \(\rho(t_1)\cap\rho(t'_1)=\emptyset\) for \(t_1,t'_1\in T_1\), and interventions on all sets \(S\subset T^1\) can be expressed as interventions on the target \(\mc{C}^2\) (i.e., the more fine-grained representations), while this is reversed in our case so that those two settings are dual to each other. We do not pursue this here any further, as those transformations are of more limited interest and applicability. Let us emphasise nevertheless that it seems ambitious to handle all cases in one framework. Indeed, combining variables in a summary variable or splitting variables in a more fine-grained description are
	meaningful operations, but it is less clear to interpret in a causal manner a definition of a transformation \((X_1,X_2)\to(Y_1,Y_2)\) that allows both at the same time. For example, intervening on \(X_1\), in general, then does not correspond to a meaningful causal operation on the variables \((Y_1,Y_2)\). We also remark that this attempt has not been made in the SCM literature, where the focus is almost exclusively on abstractions. 
	
	\section{Comparison with Abstraction in the SCM framework}\label{sec:comparison}
	\citet{rubenstein2017causal} gives the definition of \textit{exact transformations} between SCMs. While being the seminal work on the theory of causal abstractions, it is probably also the most relevant to compare to our proposals. We first recall some essential aspects of their definition of SCMs (or SEMs, for structural equation models, by their nomenclature)\footnote{In this section, some imported notations might clash with ours; the clashes are restricted to this section and should not cause any confusion.}. 
	\begin{definition}[{\citep[Definition 1]{rubenstein2017causal}}]\label{def:sem}
		Let \(\mathbb{I}_X\) be an index set. An SEM \(\mathcal{M}_X\) over variables \(X=(X_i:i\in\mathbb{I}_X\) taking values in \(\mathcal{X}\) is a tuple \((\mathcal{S}_X,\mathbb{P}_E)\), where
		\begin{itemize}
			\item \(\mathcal{S}_X\) is a set of structural equations, i.e. the set of equations \(X_i=f_i(X,E_i)\) for \(i\in\mathbb{I}_X\);
			\item \(\mathbb{P}_E\) is a distribution over the exogenous variables \(E=(E_i:i\in\mathbb{I}_X)\). 
		\end{itemize}
	\end{definition}
	Note that their definition of SCMs is a bit more general than standard ones in the literature (e.g. \citep[p.83, Definition 6.2]{peters2017elements}), in that they allow, for example, cycles and latent confounders, but they simply insist that there must be a unique solution to any interventions. They also consider a specific set of \say{allowed interventions}, rather than considering all possible interventions. We also recall some essential aspects of the notion of exact transformations. 
	\begin{definition}[{\citep[Definition 3]{rubenstein2017causal}}]\label{def:exact_transformations}
		Let \(\mathcal{M}_X\) and \(\mathcal{M}_Y\) be SCMs, and \(\tau:\mathcal{X}\rightarrow\mathcal{Y}\) a function. We say that \(\mathcal{M}_Y\) is an \textit{exact \(\tau\)-transformation} of \(\mathcal{M}_X\) if, there exists a surjective mapping \(\omega\) of the interventions such that for any intervention \(i\), \(\mathbb{P}^i_{\tau(X)}=\mathbb{P}^{\omega(i)}_Y\). 
	\end{definition}
	Note that this definition is trying to capture the same concept as our notion of interventional consistency given in (\ref{eq:interventional_consistency}): that interventions and transformations commute. However, there are several aspects in which our proposal is more appealing.
	\begin{itemize}
		\item They only consider deterministic maps \(\tau:\mathcal{X}\rightarrow\mathcal{Y}\), whereas we allow the map \(\rho\) to be stochastic.
		\item They have to find a separate map \(\omega\) \textit{between the interventions themselves}, whereas our map \(\rho\) also determines the transformation of the causal kernels.
		\item By insisting on surjectivity of \(\omega\), they only allow the consideration of abstraction, whereas we can consider more general transformations of causal spaces, such as inclusions considered in Example \ref{ex:incl}. 
	\end{itemize}
	Nevertheless, restricted to considerations amenable to both approaches, the notions coincide. For example, we return to Example \ref{ex:abs}, where we already showed that \(f_*\mathbb{P}=\mathbb{Q}\), \(f_*K_{\{X_1,X_2\}}=L_X\) and \(f_*K_{\{Y_1,Y_2\}}=L_Y\), which implies that two-variable SCM is an exact transformation of the four-variable SCM according to Definition \ref{def:exact_transformations}. 
	
	Finally, we mention that \citet{beckers2019abstracting} criticise exact transformations of \citet{rubenstein2017causal} on the basis that probabilities and allowed interventions can mask significant differences between SCMs, and then proceed to propose definitions of abstractions that depend only on the structural equations, independently of probabilities. We remark that this criticism is not valid in our framework, in that the interventional consistency of our transformations is imposed independently of probabilities, making it impossible to mask them with the choice of probability measures. That this is possible with SCMs is an artifact of the fact that in SCMs, the observational and interventional measures are coupled through the exogenous distribution, whereas in causal spaces they are completely decoupled. Moreover, we consider all possible interventions rather than a reduced set of allowed interventions. We also remark that, since probabilities and causal kernels are the primitive objects in our framework, rather than being derived by other primitive objects (namely the structural equations), it does not make sense for the transformation to be defined independently of probabilities, as done by \citet{beckers2019abstracting}. 
	
	\section{Further Properties of Causal Transformations}\label{sec:results}
	In this section we investigate various properties of causal transformations
	and connect them to the notions introduced in Section~\ref{sec:preliminaries}.
	\subsection{Existence and Uniqueness of Causal Transformations}\label{subsec:exuniqmaps}
	First, we have the following lemma on the composition of causal transformations. Recall that for two probability kernels \(\kappa_1:\Omega^1\times \mc{H}^2\to[0,1]\) mapping  \((\Omega^1,\mc{H}^1)\) to \((\Omega^2,\mc{H}^2)\) and \(\kappa_2:\Omega^2\times\mc{H}^3\to[0,1]\) mapping \((\Omega^2,\mc{H}^2)\) to \((\Omega^3,\mc{H}^3)\) the concatenation defined by \citep[p.39]{cinlar2011probability}
	\[\kappa_1\circ \kappa_2(\omega_1, A) = \int \kappa_1(\omega_1, \d\omega_2) \kappa_2(\omega_2,A)\]
	defines a probability kernel from \((\Omega^1,\mc{H}^1)\) to \((\Omega^3,\mc{H}^3)\).
	\begin{restatable}[Compositions of Causal Transformations]{lemma}{composition}\label{lem:composition}
		Let \((\kappa_1,\rho_1):\mc{C}^1\to\mc{C}^2\) and \((\kappa_2,\rho_2):\mc{C}^2\to\mc{C}^3\) be causal transformations.
		If $(\kappa_1, \rho_1)$ is an abstraction then 
		$(\kappa_3, \rho_3)=(\kappa_1\circ \kappa_2, \rho_1\circ \rho_2):\mc{C}^1\to\mc{C}^3$ is a causal transformation.
	\end{restatable}
	The proof can be found in Appendix~\ref{app:proofs2}. We remark that, unfortunately, we cannot remove the assumption that the first transformation is an abstraction. Let us clarify this through an example. 
	\begin{example}\label{ex:composition_abstraction}
		Consider an SCM with equations
		\begin{alignat*}{2}
			X_1&=N_1,\\
			X_2&=N_2,\\
			Y&=X_1+X_2+N_Y
		\end{alignat*}
		where \(N_1\), \(N_2\), and \(N_Y\) follow independent standard normal distributions. Then we can consider the causal space \(\mc{C}^1\) containing \((X_1,Y)\), the causal space \(\mc{C}^2\) containing \((X_1,X_2,Y)\) and an abstraction \(\mc{C}^3\) containing \((X_1+X_2,Y)\). Then we can embed \(\mc{C}^1\to\mc{C}^2\) and there is an abstraction \(\mc{C}^2\to\mc{C}^3\) which are both transformations of causal spaces. 
		
		However, their concatenation is not a causal transformation because it is not even admissible (and also interventional consistency does not hold for
		the intervention \(K^3_{X}\) as this cannot be expressed by \(K^1_{X_1}\)). Note that \(P_{X_2|X_1=x_1,Y=y}=N((y-x_1)/2,1/2)\) and therefore we have \(\kappa_1((x_1,y),\cdot)=\delta_{x_1}\otimes N((y-x_1)/2,1/2)\otimes\delta_{y}\).
		
		We also have \(\kappa_2((x_1,x_2,y),\cdot)=\delta_{(x_1+x_2,y)}\). Thus, their concatenation is given by
		\[\kappa_3((x_1,y),\cdot)=N((y+x_1)/2,1/2)\otimes\delta_y.\]
		So the first coordinate is not measurable with respect to \(\mc{H}^1_{X_1}\).
	\end{example}
	This shows that we lose measurability along the concatenation because the variables added in the  more complete description \(\mc{C}^2\) may depend on all other variables. 
	
	Let us now generalise Example~\ref{ex:sub} to general SCMs.
	\begin{restatable}[Inclusion of SCMs]{lemma}{scm}\label{lem:scm}
		Consider an acyclic SCM on endogenous variables \((X_1,\ldots,X_d)\in\R^d\) with observational distribution \(\P\). Let \(S\subset[d]\), \(R=S^c=[d]\setminus S\) and consider causal spaces \(\mc{C}^1=(\Omega^1,\mc{H}^1,\P^S,\K)\) and \(\mc{C}^2=(\Omega^2,\mc{H}^2,\P,\L)\), where we have \((\Omega^1,\mc{H}^1)=(\mathbb{R}^{|S|},\mathscr{B}(\mathbb{R}^{|S|}))\) and \((\Omega^2,\mathscr{H}^2)=(\mathbb{R}^d,\mathscr{B}(\mathbb{R}^d))\). Moreover, \(\P^S\) is the marginal distribution on the variables in \(S\), and the causal mechanisms \(\K\) and \(\L\) are derived from the SCM. In particular, \(\K\) is a marginalisation of \(\L\), namely, for any \(\omega\in\Omega^2\), any event \(A\in\mathscr{H}^1\) and any \(S'\subseteq S\), we have that \(K_{S'}(\omega,A)=L_{S'}(\omega,A)\). 
		
		Consider the map \(\rho:S\hookrightarrow[d]\) and \(\kappa(\cdot,A)=\P_{\mathscr{H}^1}(A)\). Then \((\rho,\kappa)\) is a causal transformation from \(\mc{C}^1\) to \(\mc{C}^2\).
	\end{restatable}
	The proof can be found in Appendix~\ref{app:proofs2}.
	
	We now investigate to what degree distributional and interventional consistency determines the causal structure on the target space. We show that generally the causal structure on \(\mc{H}_{\rho(T)}^2\) is quite rigid.
	\begin{restatable}[Rigidity of target causal structure]{lemma}{rigidity}\label{lem:rigidity}
		Let ${\mc{C}}^2=(\Omega^2,\mc{H}^2,\P^2,{\K}^2)$ and $\tilde{\mc{C}}^2=(\Omega^2,\mc{H}^2,\tilde{\P}^2,\tilde{\K}^2)$
		be two causal spaces with the same underlying measurable space. 
		Let $(\kappa, \rho)$ be an admissible pair
		for the measurable spaces $(\Omega^1,\mc{H}^1)$ and $(\Omega^2,\mc{H}^2)$.
		Assume that the pair $(\kappa, \rho)$ defines causal transformations
		$\p:\mc{C}^1\to \mc{C}^2$ 
		and $\tilde{\p}:\mc{C}^1\to\tilde{\mc{C}}^2$
		be 
		a causal transformations. 
		Then $\P^2=\tilde{\P}^2$,
		and for all $A\in \mc{H}^2_{\rho(T^1)}$
		and any $S\subset T^2$
		\begin{align*}
			K^2_S(\omega,A)
			= \tilde{K}^2_S(\omega,A)
			\quad \text{for $\P^2=\tilde{\P}^2$
				a.\ e.\ $\omega\in \Omega^2$}.
		\end{align*} 
	\end{restatable}
	The proof is in Appendix~\ref{app:proofs2}.
	We cannot expect to derive much stronger results for general causal transformation because interventional consistency does not restrict
	$K^2(\omega,A)$ for $\omega$ not in the support of $\P^2$
	or $A\notin \mc{H}_{\rho(T)}^2$.
	E.g., in the setting of Example~\ref{ex:incl}
	the causal structure on the second factor is arbitrary. 
	
	However, when we consider deterministic maps \((f,\rho)\) such that \(f:(\Omega^1,\mc{H}^1)\to (\Omega^2,\mc{H}^2)\) and \(\rho\) are surjective, then there is at most one causal structure on the target space \((\Omega^2,\mc{H}^2)\) such that the pair \((f,\rho)\) is a causal transformation (and thus a perfect abstraction).
	\begin{restatable}[Surjective Deterministic Maps]{lemma}{pushforward}\label{lem:ex_uniq_pushforward}
		Suppose \((f,\rho)\) is an admissible pair for the causal space \(\mc{C}^1=(\Omega^1,\mc{H}^1,\P^1,\K^1)\) to the measurable space \(X^2=(\Omega^2,\mc{H}^2)\) and assume that \(\rho\) is surjective and \(f:\Omega_1\to\Omega_2\) measurable. If \(f\) is surjective, there exists at most one causal space \(\mc{C}^2=(\Omega^2,\mc{H}^2,\P^2,\K^2)\) such that \((f,\rho):\mc{C}^1\to\mc{C}^2\) is a causal transformation.
		
		If, in addition, \(K^1_{\rho^{-1}(S^2)}(\cdot,A)\) is measurable with respect to \(f^{-1}(\mc{H}^2_{S^2})\) for all \(A\in f^{-1}(\mc{H}^2)\) and all \(S^2\subset T^2\)
		then a unique causal space \(\mc{C}^2\) exists such that \((f,\rho):\mc{C}^1\to\mc{C}^2\) is a causal transformation.
	\end{restatable}
	The proof is in Appendix~\ref{app:proofs2}.
	To motivate the measurability condition for $K^1(\cdot, A)$
	we remark that interventional consistency requires $K^1(\omega, A)=K^1(\omega',A)$ for 
	$\omega$, $\omega'$ with $f(\omega)=f(\omega')$ and the measurability condition in the result is a slightly stronger condition than this.
	
	% \snote{Actually we also need to show the $A\cap B$ property!!}
	% \snote{I think we should just add the measurability assumption plus factorisation lemma to the assumptions. It does not sound to bad and we need existence later on.}
	% \snote{Old Text: To prove existence of $\K^2$ it remains to verify measurability of
		% $K_S(\cdot, A)$ which is not implied by the assumptions of the lemma.
		% Under suitable assumptions this can be shown, e.g., when the state space is finite
		% or when the state spaces are $\R^{d_1}$ and $\R^{d_2}$ and $f$ is a differentiable function with non-vanishing differential. Note that by the factorisation lemma we can equivalently show that $K(\cdot, f^{-1}(A))$ is measurable with respect to $f^{-1}(\mc{H}^2)$.}
	
	Next, we show that interventions on a space can be pushed forward along a perfect abstraction.
	\begin{restatable}[Perfect Abstraction on Intervened Spaces]{lemma}{intervention}\label{lem:intervention_spaces}
		Let \(\mc{C}^1=(\Omega^1,\mc{H}^1,\P^1,\K^1)\) with \((\Omega^1,\mc{H}^1)\) a product with index set \(T^1\) and \(\mc{C}^2=(\Omega^2,\mc{H}^2,\P^2,\K^2)\) with \((\Omega^2,\mc{H}^2)\) a product with index set \(T^2\) be causal spaces, and let \((f,\rho):\mc{C}^1\to\mc{C}^2\) be a perfect abstraction. 
		
		Let \(U^1=\rho^{-1}(U^2)\subseteq T^1\) for some \(U^2\subseteq T^2\). Let \(\Q^1\) be a probability measure on \((\Omega^1,\mc{H}^1_{U^1})\)
		and \(\L^1\) a causal mechanism on \((\Omega^1,\mc{H}^1_{U^1},\Q^1)\). Suppose that, for all \(S\subseteq U^2\) and \(A\in\mathscr{H}^1\), the map \(L_{\rho^{-1}(S)}^1(\cdot,A)\) is measurable with respect to \(f^{-1}(\mc{H}^2_S)\), and consider the intervened causal spaces
		\begin{alignat*}{2}
			\mc{C}^1_I&=(\Omega^1,\mc{H}^1,(\P^1)^{\textnormal{do}(U^1,\Q^1)},(\K^1)^{\textnormal{do}(U^1,\Q^1,\L^1)}),\\
			\mc{C}^2_I&=(\Omega^2,\mc{H}^2,(\P^2)^{\textnormal{do}(U^2,\Q^2)}, (\K^2)^{\textnormal{do}(U^2,\Q^2,\L^2)}),
		\end{alignat*}
		where \(\Q^2=f_\ast\Q^1\) and \(\L^2\) is the unique family of kernels satisfying \(L^2_{S}(f(\omega),A)=L^1_{\rho^{-1}(S)}(\omega,f^{-1}(A))\)
		for all \(\omega\in\Omega^1\), \(A\in\mc{H}^2\), and \(S\subseteq U^2\). Then \((f,\rho):\mc{C}^1_I\to\mc{C}^2_I\) is a perfect abstraction.
	\end{restatable}
	The proof of this result is in Appendix~\ref{app:proofs2}.
	\subsection{Sources and causal effects under causal transformations}\label{subsec:causal_effect}
	We now study whether causal effects in target and domain of a causal transformation can be related.
	% Note that Example~\ref{ex:trivial} shows that for general stochastic maps, we cannot expect a general relation of causal effect in either direction. 
	Our first results shows that
	% in contrast
	for perfect abstractions  having no causal effect 
	in the domain implies that there is also no causal effect in the target.
	\begin{restatable}[Perfect Abstraction and No Causal Effect]{lemma}{nocausaleffect}\label{lem:no_causal_effect}
		Let \((f,\rho):\mc{C}^1\to \mc{C}^2\) be a perfect abstraction. Consider two sets \(U^2,V^2\subset T^2\) and denote \(U^1=\rho^{-1}(U^2)\) and \(V^1=\rho^{-1}(V^2)\). If \(\mc{H}^1_{U^1}\) has no causal effect on \(\mc{H}^1_{V^1}\) in \(\mc{C}^1\), then \(\mc{H}^2_{U^2}\) has no causal effect on \(\mc{H}^2_{V^2}\) in \(\mc{C}^2\).
	\end{restatable}
	The proof can be found in Appendix~\ref{app:proofs3}.
	
	On the other hand, we can show that when there is an active causal effect in the target space, there is also an active causal effect in the domain.
	\begin{restatable}[Perfect Abstraction and Active Causal Effects]{lemma}{activecausaleffect}\label{lem:active_effect}
		Let \((f,\rho):\mc{C}^1\to\mc{C}^2\) be a perfect abstraction. Consider two sets \(U^2,V^2\subset T^2\) and denote \(U^1=\rho^{-1}(U^2)\) and \(V^1=\rho^{-1}(V^2)\). Assume that \(\mc{H}^2_{U^2}\) has an active causal effect on \(\mc{H}^2_{V^2}\) in \(\mc{C}^2\). Then \(\mc{H}^1_{U^1}\) has an active causal effect on \(\mc{H}^1_{V^1}\) in \(\mc{C}^1\).
	\end{restatable}
	The proof is in Appendix~\ref{app:proofs3}.
	
	The reverse statements are not true, i.e., if there is no
	causal effect in the target there might be a causal effect in the domain and if there is an active causal effect in the domain this does not imply that ther is a causal effect in the target, which can be seen by considering a target space with only a single point.
	We can also study causal effects in the context of embedding transformations, as in Lemma~\ref{lem:scm}.
	Then we see directly  that active causal effects are preserved. On the other hand, it is straightforward to construct examples where there is no causal effect in a subsystem, but 
	there is a causal effect in a larger system. This can be achieved by a violation of faithfulness. 
	\begin{example}
		Consider the SCM
		\begin{alignat*}{2}
			X&=N_X,\\
			M&=N_X+N_M,\\
			Y&=M-X+N_Y.
		\end{alignat*}
		Then there is no causal effect from \(\sigma(X)\) to \(\sigma(Y)\) in the system \((X,Y)\) but there is a causal effect in the complete system.
	\end{example}
	Finally, we show that similar results can be established for sources. Indeed, perfect abstraction preserve sources in the following sense.
	\begin{restatable}[Perfect Abstraction and Sources]{lemma}{sources}\label{lem:sources}
		Let \((f,\rho):\mc{C}^1\to\mc{C}^2\) be a perfect abstraction. Consider two sets \(U^2,V^2\subset T^2\) and denote \(U^1=\rho^{-1}(U^2)\) and \(V^1=\rho^{-1}(V^2)\). Assume that \(\mc{H}^1_{U^1}\) is a local source of \(\mc{H}^1_{V^1}\) in \(\mc{C}^1\). Then \(\mc{H}^2_{U^2}\) is a local source of \(\mc{H}^2_{V^2}\) in \(\mc{C}^2\).
		
		In particular, this implies that if \(\mc{H}^1_{U^1}\) is a global source then \(\mc{H}^2_{U^2}\) also is a global source.
	\end{restatable}
	The proof is in Appendix~\ref{app:proofs3}.
	
	Similar to our results for causal effects, the existence of sources in the abstracted space does not guarantee the existence of sources in the domain space. Note that local sources are preserved in the setting of Lemma~\ref{lem:scm}. On the other hand, global sources are clearly not preserved, as we can add a global source to the system.
	
	\section{Conclusion}\label{sec:conclusion}
	In this paper, we developed the theory of causal spaces initiated by \citet{park2023measure} by proposing the notions of products of causal spaces and transformations of causal spaces. They are defined via natural extensions of the notions of products and probability kernels in probability theory. Not only are they mathematically elegant objects, but they have natural and important semantic interpretations as causal independence and abstraction or inclusion of causal spaces. Moreover, we explore the connections of these notions to those of causal effects and sources introduced in \citep{park2023measure}. 
	
	Despite the beauty and practical usefulness of the structural causal model and potential outcomes frameworks, we believe that the theory of causal spaces has the potential to overcome some of the longstanding limitations of them in terms of rigour and expressiveness, and the contribution of this paper is to develop the theory further in terms of treating multiple causal spaces through products and transformations rather than focusing the investigation to single causal spaces.
	
	Although probability theory does not seem to be so amenable to a category-theoretic treatment as other mathematical objects, there have been some efforts to do so \citep{lynn2010categories,adachi2016category,cho2019disintegration,fritz2020synthetic}. As future work, it would be interesting to explore extensions of the transformations proposed here to formal category-theoretic morphisms between causal spaces. 
	
	% \begin{contributions} % will be removed in pdf for initial submission 
		% 					  % (without ‘accepted’ option in \documentclass)
		%                       % so you can already fill it to test with the
		%                       % ‘accepted’ class option
		%     Briefly list author contributions. 
		%     This is a nice way of making clear who did what and to give proper credit.
		%     This section is optional.
		
		%     H.~Q.~Bovik conceived the idea and wrote the paper.
		%     Coauthor One created the code.
		%     Coauthor Two created the figures.
		% \end{contributions}
	
	% \begin{acknowledgements} % will be removed in pdf for initial submission,
		% 						 % (without ‘accepted’ option in \documentclass)
		%                          % so you can already fill it to test with the
		%                          % ‘accepted’ class option
		%     Briefly acknowledge people and organizations here.
		
		%     \emph{All} acknowledgements go in this section.
		% \end{acknowledgements}
	
	\begin{acknowledgements} % will be removed in pdf for initial submission,
		% (without ‘accepted’ option in \documentclass)
		% so you can already fill it to test with the
		% ‘accepted’ class option
		This work was supported by the Tübingen AI Center.
	\end{acknowledgements}
	% References
	\bibliography{ref}

\begin{thebibliography}{37}
\providecommand{\natexlab}[1]{#1}
\providecommand{\url}[1]{\texttt{#1}}
\expandafter\ifx\csname urlstyle\endcsname\relax
  \providecommand{\doi}[1]{doi: #1}\else
  \providecommand{\doi}{doi: \begingroup \urlstyle{rm}\Url}\fi

\bibitem[Adachi and Ryu(2016)]{adachi2016category}
Takanori Adachi and Yoshihiro Ryu.
\newblock A {C}ategory of {P}robability {S}paces.
\newblock \emph{arXiv preprint arXiv:1611.03630}, 2016.

\bibitem[Beckers and Halpern(2019)]{beckers2019abstracting}
Sander Beckers and Joseph~Y Halpern.
\newblock Abstracting {C}ausal {M}odels.
\newblock In \emph{Proceedings of the aaai conference on artificial
  intelligence}, volume~33, pages 2678--2685, 2019.

\bibitem[Beckers et~al.(2020)Beckers, Eberhardt, and
  Halpern]{beckers2020approximate}
Sander Beckers, Frederick Eberhardt, and Joseph~Y Halpern.
\newblock Approximate {C}ausal {A}bstractions.
\newblock In \emph{Uncertainty in artificial intelligence}, pages 606--615.
  PMLR, 2020.

\bibitem[Bongers et~al.(2021)Bongers, Forr{\'e}, Peters, and
  Mooij]{bongers2021foundations}
Stephan Bongers, Patrick Forr{\'e}, Jonas Peters, and Joris~M Mooij.
\newblock Foundations of {S}tructural {C}ausal {M}odels with {C}ycles and
  {L}atent {V}ariables.
\newblock \emph{The Annals of Statistics}, 49\penalty0 (5):\penalty0
  2885--2915, 2021.

\bibitem[Chalupka et~al.(2015)Chalupka, Perona, and
  Eberhardt]{chalupka2015visual}
Krzysztof Chalupka, Pietro Perona, and Frederick Eberhardt.
\newblock Visual causal feature learning.
\newblock In \emph{Proceedings of the Thirty-First Conference on Uncertainty in
  Artificial Intelligence}, UAI'15, page 181–190, Arlington, Virginia, USA,
  2015. AUAI Press.
\newblock ISBN 9780996643108.

\bibitem[Chalupka et~al.(2016)Chalupka, Bischoff, Perona, and
  Eberhardt]{chalupka2016unsupervised}
Krzysztof Chalupka, Tobias Bischoff, Pietro Perona, and Frederick Eberhardt.
\newblock Unsupervised discovery of el nino using causal feature learning on
  microlevel climate data.
\newblock In \emph{Proceedings of the Thirty-Second Conference on Uncertainty
  in Artificial Intelligence}, UAI'16, page 72–81, Arlington, Virginia, USA,
  2016. AUAI Press.
\newblock ISBN 9780996643115.

\bibitem[Chalupka et~al.(2017)Chalupka, Eberhardt, and
  Perona]{chalupka2017causal}
Krzysztof Chalupka, Frederick Eberhardt, and Pietro Perona.
\newblock Causal feature learning: an overview.
\newblock \emph{Behaviormetrika}, 44\penalty0 (1):\penalty0 137--164, 2017.
\newblock \doi{10.1007/s41237-016-0008-2}.
\newblock URL \url{https://doi.org/10.1007/s41237-016-0008-2}.

\bibitem[Cho and Jacobs(2019)]{cho2019disintegration}
Kenta Cho and Bart Jacobs.
\newblock Disintegration and {B}ayesian {I}nversion via {S}tring {D}iagrams.
\newblock \emph{Mathematical Structures in Computer Science}, 29\penalty0
  (7):\penalty0 938--971, 2019.

\bibitem[Cinlar(2011)]{cinlar2011probability}
Erhan Cinlar.
\newblock \emph{Probability and {S}tochastics}, volume 261.
\newblock Springer, 2011.

\bibitem[Durrett(2019)]{durrett2019probability}
Rick Durrett.
\newblock \emph{Probability: {T}heory and {E}xamples}, volume~49.
\newblock Cambridge university press, 2019.

\bibitem[Fritz(2020)]{fritz2020synthetic}
Tobias Fritz.
\newblock A {S}ynthetic {A}pproach to {M}arkov {K}ernels, {C}onditional
  {I}ndependence and {T}heorems on {S}ufficient {S}tatistics.
\newblock \emph{Advances in Mathematics}, 370:\penalty0 107239, 2020.

\bibitem[Geiger et~al.(2021)Geiger, Lu, Icard, and Potts]{geiger2021causal}
Atticus Geiger, Hanson Lu, Thomas Icard, and Christopher Potts.
\newblock Causal {A}bstractions of {N}eural {N}etworks.
\newblock \emph{Advances in Neural Information Processing Systems},
  34:\penalty0 9574--9586, 2021.

\bibitem[Geiger et~al.(2023)Geiger, Potts, and Icard]{geiger2023causal}
Atticus Geiger, Chris Potts, and Thomas Icard.
\newblock Causal {A}bstraction for {F}aithful {M}odel {I}nterpretation.
\newblock \emph{arXiv preprint arXiv:2301.04709}, 2023.

\bibitem[Halpern(2000)]{halpern2000axiomatizing}
Joseph~Y Halpern.
\newblock Axiomatizing {C}ausal {R}easoning.
\newblock \emph{Journal of Artificial Intelligence Research}, 12:\penalty0
  317--337, 2000.

\bibitem[Hern\`an and Robins(2020)]{hernan2020what}
MA~Hern\`an and JM~Robins.
\newblock \emph{Causal {I}nference: {W}hat {I}f}.
\newblock Boca Raton: Chapman \& Hall/CRC, 2020.

\bibitem[Ibeling and Icard(2023)]{ibeling2023comparing}
Duligur Ibeling and Thomas Icard.
\newblock Comparing {C}ausal {F}rameworks: {P}otential {O}utcomes, {S}tructural
  {M}odels, {G}raphs, and {A}bstractions.
\newblock \emph{arXiv preprint arXiv:2306.14351}, 2023.

\bibitem[Illari et~al.(2011)Illari, Russo, and Williamson]{illari2011causality}
Phyllis~McKay Illari, Federica Russo, and Jon Williamson.
\newblock \emph{Causality in the {S}ciences}.
\newblock Oxford University Press, 2011.

\bibitem[Imbens(2019)]{imbens2019potential}
Guido Imbens.
\newblock Potential {O}utcome and {D}irected {A}cyclic {G}raph {A}pproaches to
  {C}ausality: {R}elevance for {E}mpirical {P}ractice in {E}conomics.
\newblock Technical report, National Bureau of Economic Research, 2019.

\bibitem[Imbens and Rubin(2015)]{imbens2015causal}
Guido~W Imbens and Donald~B Rubin.
\newblock \emph{Causal {I}nference in {S}tatistics, {S}ocial, and {B}iomedical
  sciences}.
\newblock Cambridge University Press, 2015.

\bibitem[Kekić et~al.(2023)Kekić, Schölkopf, and
  Besserve]{kekic2023targeted}
Armin Kekić, Bernhard Schölkopf, and Michel Besserve.
\newblock Targeted reduction of causal models, 2023.

\bibitem[Lynn(2010)]{lynn2010categories}
MELISSA Lynn.
\newblock Categories of {P}robability {S}paces, 2010.

\bibitem[Massidda et~al.(2023)Massidda, Geiger, Icard, and
  Bacciu]{massidda2023causal}
Riccardo Massidda, Atticus Geiger, Thomas Icard, and Davide Bacciu.
\newblock Causal {A}bstraction with {S}oft {I}nterventions.
\newblock In \emph{2nd Conference on Causal Learning and Reasoning}, 2023.

\bibitem[Otsuka and Saigo(2022)]{otsuka2022equivalence}
Jun Otsuka and Hayato Saigo.
\newblock On the {E}quivalence of {C}ausal {M}odels: {A} {C}ategory-{T}heoretic
  {A}pproach.
\newblock In \emph{Conference on Causal Learning and Reasoning}, pages
  634--646. PMLR, 2022.

\bibitem[Otsuka and Saigo(2023)]{otsuka2023process}
Jun Otsuka and Hayato Saigo.
\newblock Process {T}heory of {C}ausality: a {C}ategory-{T}heoretic
  {P}erspective.
\newblock \emph{Behaviormetrika}, pages 1--16, 2023.

\bibitem[Park et~al.(2023)Park, Buchholz, Sch{\"o}lkopf, and
  Muandet]{park2023measure}
Junhyung Park, Simon Buchholz, Bernhard Sch{\"o}lkopf, and Krikamol Muandet.
\newblock A measure-theoretic axiomatisation of causality.
\newblock In \emph{Thirty-seventh Conference on Neural Information Processing
  Systems}, 2023.

\bibitem[Pearl(2009)]{pearl2009causality}
Judea Pearl.
\newblock \emph{Causality}.
\newblock Cambridge university press, 2009.

\bibitem[Pearl and Mackenzie(2018)]{pearl2018book}
Judea Pearl and Dana Mackenzie.
\newblock \emph{The {B}ook of {W}hy}.
\newblock Basic Books, New York, 2018.

\bibitem[Peters et~al.(2017)Peters, Janzing, and
  Sch{\"o}lkopf]{peters2017elements}
Jonas Peters, Dominik Janzing, and Bernhard Sch{\"o}lkopf.
\newblock \emph{Elements of {C}ausal {I}nference: {F}oundations and {L}earning
  {A}lgorithms}.
\newblock The MIT Press, 2017.

\bibitem[Rischel and Weichwald(2021)]{rischel2021compositional}
Eigil~F Rischel and Sebastian Weichwald.
\newblock Compositional {A}bstraction {E}rror and a {C}ategory of {C}ausal
  {M}odels.
\newblock In \emph{Uncertainty in Artificial Intelligence}, pages 1013--1023.
  PMLR, 2021.

\bibitem[Rubenstein et~al.(2017)Rubenstein, Weichwald, Bongers, Mooij, Janzing,
  Grosse-Wentrup, and Sch{\"o}lkopf]{rubenstein2017causal}
P~Rubenstein, S~Weichwald, S~Bongers, J~Mooij, D~Janzing, M~Grosse-Wentrup, and
  B~Sch{\"o}lkopf.
\newblock Causal {C}onsistency of {S}tructural {E}quation {M}odels.
\newblock In \emph{33rd Conference on Uncertainty in Artificial Intelligence
  (UAI 2017)}, pages 808--817. Curran Associates, Inc., 2017.

\bibitem[Russo(2010)]{russo2010causality}
Federica Russo.
\newblock \emph{Causality and {C}ausal {M}odelling in the {S}ocial {S}ciences}.
\newblock Springer, 2010.

\bibitem[Sadeghi and Soo(2023)]{sadeghi2023axiomatization}
Kayvan Sadeghi and Terry Soo.
\newblock Axiomatization of {I}nterventional {P}robability {D}istributions.
\newblock \emph{arXiv preprint arXiv:2305.04479}, 2023.

\bibitem[Woodward(2005)]{woodward2005making}
James Woodward.
\newblock \emph{Making {T}hings {H}appen: {A} {T}heory of {C}ausal
  {E}xplanation}.
\newblock Oxford university press, 2005.

\bibitem[Xia and Bareinboim(2024)]{xia2024neural}
Kevin Xia and Elias Bareinboim.
\newblock Neural {C}ausal {A}bstractions.
\newblock \emph{arXiv preprint arXiv:2401.02602}, 2024.

\bibitem[Ze{\v{c}}evi{\'c} et~al.(2023)Ze{\v{c}}evi{\'c}, Willig, Busch, and
  Seng]{zevcevic2023continual}
Matej Ze{\v{c}}evi{\'c}, Moritz Willig, Florian~Peter Busch, and Jonas Seng.
\newblock Continual {C}ausal {A}bstractions.
\newblock In \emph{AAAI Bridge Program on Continual Causality}, pages 45--51.
  PMLR, 2023.

\bibitem[Zennaro(2022)]{zennaro2022abstraction}
Fabio~Massimo Zennaro.
\newblock Abstraction between {S}tructural {C}ausal {M}odels: {A} {R}eview of
  {D}efinitions and {P}roperties.
\newblock In \emph{UAI 2022 Workshop on Causal Representation Learning}, 2022.

\bibitem[Zennaro et~al.(2023)Zennaro, Dr{\'a}vucz, Apachitei, Widanage, and
  Damoulas]{zennaro2023jointly}
Fabio~Massimo Zennaro, M{\'a}t{\'e} Dr{\'a}vucz, Geanina Apachitei, W~Dhammika
  Widanage, and Theodoros Damoulas.
\newblock Jointly {L}earning {C}onsistent {C}ausal {A}bstractions over
  {M}ultiple {I}nterventional {D}istributions.
\newblock \emph{arXiv preprint arXiv:2301.05893}, 2023.

\end{thebibliography}
	
	\newpage
	
	\onecolumn
	
	\title{Products, Abstractions, and Inclusions of Causal Spaces\\(Supplementary Material)}
	\maketitle
	\def\thefootnote{*}\footnotetext{Equal Contribution}\def\thefootnote{\arabic{footnote}}
	\appendix
	
	In the supplementary material, we provide the missing proofs for the results in our paper, and we recall some definitions from probability theory.
	\section{Background on Probability Theory}\label{app:background}
	Let us collect some definitions and notations from probability theory. Recall that a probability space \((\Omega,\mathscr{H},\mathbb{P})\) is defined through the following axioms:
	\begin{enumerate}[(i)]
		\item \(\Omega\) is a set; 
		\item \(\mathscr{H}\) is a collection of subsets of \(\Omega\), called \textit{events}, such that 
		\begin{itemize}
			\item \(\Omega\in\mathscr{H}\);
			\item if \(A\in\mathscr{H}\), then \(\Omega\setminus A\in\mathscr{H}\);
			\item if \(A_1,A_2,...\in\mathscr{H}\), then \(\cup_nA_n\in\mathscr{H}\);
		\end{itemize}
		\item \(\mathbb{P}\) is a probability measure on \((\Omega,\mathscr{H})\), i.e. a function \(\mathbb{P}:\mathscr{H}\rightarrow[0,1]\) satisfying
		\begin{itemize}
			\item \(\mathbb{P}(\emptyset)=0\);
			\item \(\mathbb{P}(\cup_nA_n)=\sum_n\mathbb{P}(A_n)\) for any disjoint sequence \(A_n\) in \(\mathscr{H}\);
			\item \(\mathbb{P}(\Omega)=1\).
		\end{itemize}
	\end{enumerate}
	Probability kernels (sometimes called Markov kernels) from \((\Omega^1,\mc{H}^1)\) to \((\Omega^2,\mc{H}^2)\) are maps \(\kappa:\Omega^1\times\mc{H}^2\to[0,1]\) such that
	\begin{enumerate}[(i)]
		\item For every \(A\in \mc{H}^2\) the function
		\[\omega\to\kappa(\omega,A)\]
		is measurable with respect to \(\mc{H}^1\). 
		\item For every \(\omega\in \Omega^1\) the function
		\[A\to\kappa(\omega,A)\]
		defines a measure on \((\Omega^2,\mc{H}^2)\).
	\end{enumerate}
	See, for example, \citep[p.37, Section I.6]{cinlar2011probability} for details. 
	
	We remark that we can concatenate probability kernels and we can consider product kernels. 
	
	We denote integration with respect to a (probability) measure \(\P\) by
	\[\int\P(\d\omega)T(\omega)\]
	for a measurable map \(T:(\Omega,\mc{H})\to(\R,\mc{B}(\R))\).
	
	For a measurable map \(f:(\Omega^1,\mc{H}^1)\to(\Omega^2,\mc{H}^2)\) we define the pushforward measure \(f_\ast\P\) by \(f_\ast\P(A)=\P(f^{-1}(A))\). The transformation formula for the pushforward measure reads
	\[\int f_\ast\P(\d\omega)T(\omega)=\int\P(\d\omega')T\circ f(\omega').\]
	Finally we recall the Factorisation lemma \citep[p.76, Theorem II.4.4]{cinlar2011probability}.
	\begin{lemma}[Factorisation Lemma]\label{lem:factor}
		Let $T:\Omega^1\to (\Omega^2, \mc{H}^2)$ be  a function.\
		A function $f:\Omega^1\to \R$ is $\sigma(T)$-$\mc{B}(\R)$ measurable
		if and only if there is a measurable function $g:(\Omega^2,\mc{H}^2)\to (\R,\mc{B}(\R))$ such that $f = g\circ T$.  
	\end{lemma}
	
	\section{Proofs}\label{app:proofs}
	In this appendix we collect  the proofs of all the results in the paper. 
	\subsection{Proofs for Section~\ref{sec:product_spaces}}\label{app:proofs1}
	\product*
	\begin{proof}[Proof of Lemma \ref{lem:product_space}]
		It is a standard fact that \(\K^1\otimes\K^2\) defines a family of probability kernels\footnote{See, e.g. math.stackexchange.com/questions/84078/product-of-two-probability-kernel-is-a-probability-kernel}. For the first axiom of causal kernels (Definition \ref{def:causal_space}(i)), we observe that 
		\begin{alignat*}{2}
			(K^1\otimes K^2)_{\emptyset}((\omega_1,\omega_2),A_1\times A_2)&=K^1_{\emptyset} (\omega_1, A_1)K^2_{\emptyset} (\omega_2,A_2)\\
			&=\P^1(A_1)\P^2(A_2)\\
			&=\P^1\otimes\P^2(A_1\times A_2).
		\end{alignat*}
		By standard reasoning based on the monotone class theorem, this extends to $A\in \mc{H}^1\otimes \mc{H}^2$ and therefore
		the first axiom of causal spaces is satisfied.
		For the second axiom of causal spaces, for any \(S=S^1\cup S^2\), first fix arbitrary \(A_1\in\mathscr{H}^1_{S^1}\) and \(A_2\in\mc{H}^1_{S^2}\). Then, for all \(B_1\in\mathscr{H}^1\) and \(B_2\in\mathscr{H}^2\), we find that for all \(\omega=(\omega_1,\omega_2)\),
		\begin{alignat*}{2}
			L_S(\omega,(A_1\times A_2)\cap(B_1\times B_2))&=K^1_{S_1}(\omega_1,A_1\cap B_1)K^2_{S^2}(\omega_2,A_2\cap B_2)\\
			&=\mathbf{1}_{A_1}(\omega_1)K^1_{S^1}(\omega_1,B_1)\mathbf{1}_{A_2}(\omega_2)K^2_{S^2}(\omega_2,B_2)\\
			&=\mathbf{1}_{A_1\times A_2}(\omega)L_S(\omega,B_1\times B_2).
		\end{alignat*}
		Hence, for this fixed pair \(A_1\), \(A_2\) and this \(\omega\), the measures \(B\mapsto L_S(\omega,(A_1\times A_2)\cap B)\) and \(B\mapsto\mathbf{1}_{A_1\times A_2}(\omega)L_S(\omega,B)\) are identical on the generating rectangles \(B_1\times B_2\), hence they are identical on all of \(\mathscr{H}^1\otimes\mathscr{H}^2\) by the standard monotone class theorem reasoning. Now, since this is true for arbitrary rectangles \(A_1\times A_2\)  with \(A_1\in\mathscr{H}^1_{S^1}\) and \(A_2\in\mathscr{H}^2_{S^2}\), if we now fix \(B\in\mathscr{H}^1\otimes\mathscr{H}^2\), we have that the two measures \(A\mapsto L_S(\omega,A\cap B)\) and \(A\mapsto\mathbf{1}_A(\omega)L_S(\omega,B)\) on \(\mathscr{H}^1_{S^1}\otimes\mathscr{H}^2_{S^2}\) are identical on the generating rectangles \(A_1\times A_2\), hence they are identical on all of \(\mathscr{H}^1_{S^1}\otimes\mathscr{H}^2_{S^2}\). Now both \(A\) and \(B\) are arbitrary elements of \(\mathscr{H}^1\otimes\mathscr{H}^2\) and \(\mathscr{H}^1_{S^1}\otimes\mathscr{H}^2_{S^2}\) respectively. To conclude, we have, for all \(\omega\), \(A\in\mc{H}^1\otimes\mc{H}^2\) and \(B\in\mathscr{H}^1_{S^1}\otimes\mathscr{H}^2_{S^2}\), 
		\[L_S(\omega,A\cap B)=\mathbf{1}_A(\omega)L_S(\omega,B),\]
		confirming the second axiom of causal spaces. 
	\end{proof}
	
	\productcausaleffect*
	\begin{proof}[Proof of Lemma \ref{lem:product_no_causal_effect}]
		\begin{enumerate}[(i)]
			%		\snote{I was not completely happy with this proof because
				%			we should make clear which causal mechanism we use (product or original and $A=E_T\times A$ is also a bit too abusive of notation. Should be a bit more explicit. Please check}
			\item Denote the causal kernels on the product space by $K^p$. Take any event \(A\in\mathscr{H}_{T^2}\), and any \(S\subseteq T^1\cup T^2\). Note that \(S\) can be written as a union \(S=S^1\cup S^2\) for some \(S^1\subseteq T^1\) and \(S^2\subseteq T^2\). Then see that, by writing \(A=\Omega^1\times A'\in\mathscr{H}_{T^1}\otimes\mathscr{H}_{T^2}\)
			with $A'\subseteq\Omega^2$,
			\begin{alignat*}{2}
				K_S^p(\omega,A)&=K^p_{S^1\cup S^2}(\omega,A)\\
				&=K_{S^1}^1(\omega,\Omega_1)K_{S^2}^2(\omega,A')\\
				&=K_{\emptyset}^1(\omega,\Omega_1)K_{S^2}(\omega,A')\\
				&=K_{S\setminus T^1}(\omega,A).
			\end{alignat*}
			Here we used that $K_{S^1}^1(\omega,\Omega_1)=1=K_{\emptyset}^1(\omega,\Omega_1)$ because $K(\omega,\cdot)$ is a probability measure for a probability kernel.
			So \(\mathscr{H}_{T^1}\) has no causal effect on \(A\). Implication in the other direction follows the same argument. 
			\item Take any \(A\in\mathscr{H}_{T^2}\). By (i), \(\mathscr{H}_{T^1}\) has no causal effect on \(A\), so
			\[K_{T^1}(\omega,A)=K_{T^1\setminus T^1}(\omega,A)=K_\emptyset(\omega,A)=\mathbb{P}(A).\]
			But since \(\mathscr{H}_{T^1}\) and \(\mathscr{H}_{T^2}\) are probabilistically independent, \(\mathbb{P}_{T^1}(A)=\mathbb{P}(A)\). Hence, \(\mathbb{P}_{T^1}(A)=K_{T^1}(\omega,A)\), meaning \(\mathscr{H}_{T_1}\) is a source of \(A\). Since \(A\in\mathscr{H}_{T_2}\) was arbitrary, \(\mathscr{H}_{T_1}\) is a source of \(\mathscr{H}_{T_2}\). The implication in the other direction follows the same argument. 
		\end{enumerate}
	\end{proof}
	
	\subsection{Proofs for Section~\ref{subsec:exuniqmaps}}\label{app:proofs2}
	\composition*
	\begin{proof}[Proof of Lemma \ref{lem:composition}]
		First, we claim that the pair \((\kappa_3,\rho_3)=(\kappa_1\circ\kappa_2,\rho_1\circ\rho_2)\) is admissible. We have to show that, for any \(S^3\subset\rho_3(T^1)\) and \(A\in\mc{H}^3_{S^3}\), the map \(\kappa_3(\cdot,A)\) is measurable with respect to \(\mc{H}^1_{\rho_3^{-1}(S^3)}\).
		
		Let us call \(\rho_2^{-1}(S^3)=S^2\). Note that, since \((\kappa_2,\rho_2):\mc{C}^2\to\mc{C}^3\) is a causal transformation, \(\kappa_2(\cdot,A)\) is measurable with respect to \(\mc{H}^2_{S^2}\).
		Since we assume that the first map is an abstraction, we find that \(S^2\subset \rho^1(T^1)=T^2\), and thus by Definition~\ref{def:admissible} that for \(B\in\mc{H}_{S^2}^2\) the function \(\kappa_1(\cdot,B)\) is measurable with respect to \(\mc{H}^1_{\rho^{-1}_3(S^3)}\), where we used \(\rho^{-1}_3(S^3)=\rho_1^{-1}(S^2)\). We now use the relation \(\kappa_3(\omega,A)=\int \kappa_1(\omega,\d\omega')\kappa_2(\omega',A)\). Since \(\kappa_2(\cdot,A)\) is measurable with respect to \(\mc{H}^2_{S^2}\), we conclude that we can approximate \(\kappa_2(\cdot,A)\) by a simple function \(\sum\alpha_i\bs{1}_{B_i}(\cdot)\) with \(B_i\in\mathscr{H}^2_{S^2}\). But for such a simple function, we find
		\[\int\kappa_1(\omega,\d\omega')\sum_i\alpha_i\bs{1}_{B_i}(\omega')=\sum_i\alpha_i\kappa_1(\omega,B_i),\]
		which is measurable with respect to \(\mc{H}^1_{S^1}\) as a sum of measurable functions because \((\kappa_1,\rho_1)\) is admissible. By passing to the limit \((\kappa_3 ,\rho_3)\) is admissible.
		
		Next we show that distributional consistency holds, which follows directly from distributional consistency of \((\kappa_1,\rho_1)\) and \((\kappa_2,\rho_2)\):
		\begin{alignat*}{2}
			\int\P^1(\d\omega)\kappa_3(\omega,A)&=\int\P^1(\d\omega)\kappa_1(\omega,\d\omega_2)\kappa_2(\omega_2,A)\\
			&=\int\P^2(\d\omega_2)\kappa_2(\omega_2,A)\\
			&=\P^3(A).
		\end{alignat*}
		Next we consider interventional consistency. Let \(S^3\subset\rho_3(T^1)\) and define \(S^2=\rho_2^{-1}(S^3)\) and \(S^1=\rho^{-1}_3(S^1)=\rho_1^{-1}(S^2)\). Note that, since \((\kappa_1,\rho_1)\) is an abstraction, i.e., \(\rho_1\) is surjective, we have \(S^2\subset\rho_1(T^1)=T^2\). Now we find that, for \(\omega_1\in\Omega^1\) and \(A\in\mc{H}^3\),
		\begin{alignat*}{2}
			\int\kappa_3(\omega,\d\omega')K_{S_3}^3(\omega',A)&=\int\kappa_1(\omega,\d \omega_2)\kappa_2(\omega_2,\d\omega')K_{S_3}^3(\omega',A)\\
			&=\int\kappa_1(\omega,\d\omega_2)K_{S_2}^2(\omega_2,\d\omega')\kappa_2(\omega',A)\\
			&=\int K_{S_1}^1(\omega,\d \omega')\kappa_1(\omega',\d\omega_2)\kappa_2(\omega_2,A)\\
			&=\int K_{S_1}^1(\omega,\d\omega')\kappa_3(\omega',A).
		\end{alignat*}
		This ends the proof as we have shown that $(\kappa_3, \rho_3)$ is a causal transformation.
	\end{proof}
	\scm*
	\begin{proof}[Proof of Lemma~\ref{lem:scm}]
		First we note that as in Example~\ref{ex:sub} it is clear that \((\kappa,\rho)\) is admissible and 
		\[\int\kappa(x_S,A)\P^S(d x_S)=\int\P_{\mathscr{H}^1}(A)d\P^S=\P(A),\]
		so we have distributional consistency. 
		
		For interventional consistency, let \(A\in\mathscr{H}^1\), \(S'\subseteq S\) and \(\omega\in\Omega^1\) be arbitrary. Then see that
		\begin{alignat*}{3}
			\int K_{S'}(\omega,d\omega')\kappa(\omega',A)&=\int K_{S'}(\omega,d\omega')\mathbb{P}_{\mathscr{H}^1}(\omega',A)\\
			&=\int K_{S'}(\omega,d\omega')\mathbf{1}_A(\omega')&&\text{since }A\in\mathscr{H}^1\\
			&=K_{S'}(\omega,A).
		\end{alignat*}
		On the other hand, see that
		\begin{alignat*}{3}
			\int\kappa(\omega,d\omega')L_{S'}(\omega',A)&=\int\mathbb{P}_{\mathscr{H}^1}(\omega,d\omega')L_{S'}(\omega',A)\\
			&=\int\mathbf{1}_{d\omega'}(\omega)L_{S'}(\omega',A)&&\text{since }L_{S'}(\cdot,A)\text{ is measurable with respect to }\mathscr{H}^1\\
			&=L_{S'}(\omega,A).
		\end{alignat*}
		But by the marginalisation condition on the causal mechanisms \(\K\) and \(\L\), we have that \(L_{S'}(\omega,A)=K_{S'}(\omega,A)\) for all \(\omega\in\Omega^1\). This proves interventional consistency.

		% Regarding interventional consistency, we start by denoting the parents of each \(X_i\) as \(\mathrm{PA}_i\), and the index set of \(\mathrm{PA}_i\) as \(\mathrm{pa}(i)\). We observe first that for \(S'\subset S\), we have
		% \[L_{S'}(x_S,\d x_S')=L_{S'}(x_{S'},\d x_S')=\delta_{x_{S'}}\prod_{i\in[d]\setminus S'}P_{X_i\mid\mathrm{PA}_i=\pi_{S\mathrm{pa}_i}(x_S)}(\d x_{i})\]
		% and we obtain the corresponding kernel \(K_{S'}\) by marginalisation
		% \[K_{S'}(x_S,\d x_S')=K_{S'}(x_{S'},\d x_S')=\mar_{r\in R}L_{S'}(x_{S'},\d x_S').\]
		% Now we consider $A\in \mc{B}(\R^{[d]})_S$. 
		% Note that this implies that we can write  $A=\{x: x_S\in \tilde{A}\}$.
		% The definition of $\kappa$ implies that $ \kappa(x_S,A)=\1_{\tilde{A}}(x_S)$ and so we find
		% \begin{align}
			%    \int  K_{S'}(x_S, \d x_S') \kappa(x_S',A)
			%    = K_{S'}(x_S, \tilde{A}).
			% \end{align}
		% On the other hand, we have
		% \begin{align}
			% \begin{split}
				%     \int \kappa(x_S, \d x') L_{S'}( x', A)
				%   &  = \int \left(\mar_{S'^c} \kappa(x_S, \d x')\right) L_{S'}(x'_{S'},A)
				%   \\
				%   &=\int \delta_{x_{S'}}(\d x'_{S'}) L_{S'}(x'_{S'},A)
				%    \\
				%    &=L_{S'}(x_{S'},A)=\int \mar_{S^c}L_{S'}(x_{S'},\d x') \1_{\tilde{A}}(x'_S)
				%    \\
				%    &=K_{S'}(x_S, \tilde{A}).
				%    \end{split}
			% \end{align}
		% This ends the proof.
	\end{proof}
	\rigidity*
	\begin{proof}[Proof of Lemma~\ref{lem:rigidity}]
		Applying distributional consistency of $\p$
		and $\tilde{\p}$,
		we find, for all $A\in \mc{H}^2$,
		\begin{align}
			\P^2(A)=\int \P^1(\d\omega)\kappa(\omega, A)
			=\tilde{P}^2(A)
		\end{align}
		and thus $\P^2=\tilde{\P}^2$. 
		Next, we consider $A\in \mc{H}^2_{\rho(T^1)}$
		and $S\subset \rho(T^1)$. 
		Let us define 
		\begin{align}
			B = \{\omega\in \Omega^2: K^2_S(\omega,A)<\tilde{K}_S^2(\omega,A)\}.
		\end{align}
		Since $K^2_S(\cdot, A)$ and $\tilde{K}^2_S(\cdot,A)$
		are $\mc{H}^2_S$ measurable we find $B\in \mc{H}^2_S\subset
		\mc{H}^2_{\rho(T^1)}$.
		Then the definition of causal spaces (see Definition~\ref{def:causal_space}) implies that
		\begin{align}
			K^2_S(\omega', A\cap B)=\bs{1}_B(\omega')K^2_S(\omega', A).
		\end{align}
		Note that $A\cap B\in \mc{H}^2_{\rho(T^1)}$
		so we can apply interventional consistency \eqref{eq:interventional_consistency}
		for $\mc{C}^2$ and $\tilde{\mc{C}}^2$
		and obtain for any $\omega$
		\begin{align}
			\begin{split}
				\int  \kappa(\omega, \d \omega')\bs{1}_B(\omega') K^2_S(\omega', A)
				&=
				\int  \kappa(\omega, \d \omega')\bs{1}_B(\omega') K^2_S(\omega', A\cap B)
				=\int K^1_{\rho^{-1}(S)}(\omega,\d\omega')\kappa(\omega',A)
				\\
				&= \int  \kappa(\omega, \d \omega')\bs{1}_B(\omega') \tilde{K}^2_S(\omega', A\cap B)
				= \int  \kappa(\omega, \d \omega')\bs{1}_B(\omega') \tilde{K}^2_S(\omega', A)
			\end{split}
		\end{align}
		We integrate this relation with respect to $\P^1(\d \omega)$ 
		and then apply distributional consistency and get
		\begin{align}
			\begin{split}
				0 &=  \int \P^1(\d \omega) \kappa(\omega, \d \omega') 
				\bs{1}_B(\omega') (\tilde{K}^2_S(\omega', A)-K^2_S(\omega', A))
				\\
				&=\int \P^2(\d \omega')  
				\bs{1}_B(\omega') (\tilde{K}^2_S(\omega', A)-K^2_S(\omega', A))
				\\
				&=\int_B \P^2(\d\omega')\,(\tilde{K}^2_S(\omega', A)-K^2_S(\omega', A)).
			\end{split}
		\end{align}
		On $B$ the last term is strictly positive by definition. Thus we conclude that $\P^2(B)=0$ and thus $\tilde{K}^2_S(\omega', A)\leq {K}^2_S(\omega', A)$ holds almost surely. 
		The same reasoning implies the reverse bound and we conclude that $\P^2$ almost surely the relation 
		\begin{align}
			\tilde{K}^2_S(\omega', A)={K}^2_S(\omega', A)
		\end{align}
		holds.
	\end{proof}
	% \begin{remark}
		% 	The result is in general not tight because we can actually even conclude that $\tilde{K}^2_S(\omega', A)={K}^2_S(\omega', A)$
		% 	holds for $\kappa(\omega, \cdot)$ every $\omega'$ which is more restrictive.
		% \end{remark}
	
	\pushforward*
	\begin{proof}[Proof of Lemma~\ref{lem:ex_uniq_pushforward}]
		We first prove uniqueness. The relation \(f_\ast\P^1=\P^2\) that are necessarily true for deterministic maps (see Section \ref{sec:transformations}) implies that \(\P^2\) is predetermined. 
		Moreover, we find that, by (\ref{eq:int_consistent_f}), for any \(A\in\mc{H}^2\), \(S\subset T^2\) and any \(\omega\in\Omega^1\),
		\[K^1_{\rho^{-1}(S)}(\omega,f^{-1}(A))=K^2_S(f(\omega),A).\]
		But since \(f\) is surjective we conclude that due to interventional consistency \(K^2_S(\omega',A)\) for \(\omega'\in\Omega^2\)
		is unique.
		
		To prove the existence we note that by assumption for fixed \(A\in\mc{H}^2\) the function \(K^1_{\rho^{-1}(S^2)}(\cdot,f^{-1}(A))\) is measurable with respect to \(f^{-1}(\mc{H}^2_{S^2})\). Now by the Factorisation Lemma 
		(see Lemma~\ref{lem:factor} in Appendix~\ref{app:background}) there is a
		measurable function \(g:(\Omega^2, \mc{H}^2_{S^2})\to\R\)
		such that 
		\[K^1_{\rho^{-1}(S^2)}(\omega,f^{-1}(A))=g\circ f(\omega).\]
		We define $K^2_{S^2}(\omega', A)=g(\omega')$. By surjectivity this defines
		\(K^2_{S^2}\) everywhere and this defines a probability kernel because \(g\) is measurable.
		
		It remains to verify that the resulting \(\mc{C}^2\) is indeed a causal space. Using interventional and distributional consistency we obtain
		\begin{alignat*}{2}
			K^2_\emptyset(f(\omega),A)&=K^1_\emptyset(\omega,f^{-1}(A))\\
			&=\P^1(f^{-1}(A))\\
			&=f_\ast\P^1(A)\\
			&=\P^2(A).
		\end{alignat*}
		This verifies the first property of causal spaces. For the second property we observe that for $A\in \mc{H}^2_{S^2}$, $S^1=\pi^{-1}(S^2)$ using causal consistency
		\begin{alignat*}{2}
			K^2_{S^2}(f(\omega),A\cap B)&=K^1_{S^1}(\omega,f^{-1}(A\cap B))\\
			&=K^1_{S^1}(\omega,f^{-1}(A)\cap f^{-1}(B))\\
			&=\1_{f^{-1}(A)}(\omega)K^1_{S^1}(\omega,f^{-1}(B))\\
			&=\1_{A}(f(\omega))K^2_{S^2}(f(\omega),B).
		\end{alignat*}
		Here we used that \(\mc{C}^1\) is a causal space and \(f^{-1}(A)\in\mc{H}^1_{S^1}\). Thus, we conclude that we obtained a causal space \(\mc{C}^2\).
	\end{proof}
	\intervention*
	\begin{proof}[Proof of Lemma~\ref{lem:intervention_spaces}]
		First, we note that by Lemma~\ref{lem:ex_uniq_pushforward} \(\L^2\) exists and is unique. Thus, we need to verify distributional consistency and interventional consistency. 
		
		Let us first show \(f_\ast (\P^1)^{\text{do}(U^1,\Q^1)}=(\P^2)^{\text{do}(U^2,\Q^2)}\). Since \((f,\rho)\) is a causal transformation (i.e., interventional consistency as in \eqref{eq:interventional_consistency} holds), we find that, for \(A\in\mc{H}^2\), 
		\begin{alignat*}{2}
			f_\ast(\P^1)^{\text{do}(U^1,\Q^1)}(A)&=\int\Q^1(\d\omega)K_{U^1}^1(\omega,f^{-1}(A))\\
			&=\int\Q^1(\d\omega)K_{U^2}^2(f(\omega),A)\\
			&=\int(f_\ast\Q^1)(\d\omega')K_{U^2}^2(\omega',A)\\
			&=\int\Q^2(\d\omega')K_{U^2}^2(\omega',A)\\
			&=(\P^2)^{\text{do}(U^2,\Q^2)}(A).
		\end{alignat*}
		Here we used the change of variable for pushforward-measures.
		
		Next, we show interventional consistency of \((f,\rho):\mathscr{C}^1_I\to\mathscr{C}^2_I\). For this, we introduce the shorthand \(f_S=\pi_S\circ f\). Note that since \(f_S\) is measurable with respect to \(\mc{H}^1_{\rho^{-1}(S)}\) we can find \(\tilde{f}_S\) such that \(f_S(\omega)=\tilde{f}_{S}(\omega_{\rho^{-1}(S)})\). Note that, by the interventional consistency of \((f,\rho):\mathscr{C}^1\to\mc{C}^2\), we have
		\[K_{\rho^{-1}(S)}^1(\omega,f^{-1}(S))=K^2_S(f(\omega),A)=K^2_S(\tilde{f}_{S}(\omega_S),A).\]
		We can now show for \(A\in\mc{H}^2\) and \(S^1=\rho^{-1}(S^2)\) that
		\begin{alignat*}{2}
			(K^1)_{S^1}^{\do(U^1,\Q^1,\L^1)}(\omega,f^{-1}(A))&=\int L^1_{S^1\cap U^1}(\omega_{S^1\cap U^1},\d\omega'_{U^1})
			K^1_{S^1\cup U^1}((\omega_{S^1\setminus U^1},\omega'_{U^1}), f^{-1}(A))\\
			&=\int L^1_{S^1\cap U^1}(\omega_{S^1\cap U^1},\d\omega'_{U^1})K^2_{S^2\cup U^2}(\tilde{f}_{S^2\setminus U^2}(\omega_{S^1\setminus U^1}),\tilde{f}_{U^2}(\omega'_{U^1}),A)\\
			&=\int\left((\tilde{f}_{U^1})_\ast(L^1_{S^1\cap U^1}(\omega_{S^1\cap U^1},\cdot)\right)(\d \overline{\omega}_{U^2})K^2_{S^2\cup U^2}(\tilde{f}_{S^2\setminus U^2}(\omega_{S^1\setminus U^1}),\overline{\omega}_{U^2},A)\\
			&=\int L^2_{S^2\cap U^2}(f(\omega)_{S^2\cap U^2}),\d\overline{\omega}_{U^2})K^2_{S^2\cup U^2}(f(\omega)_{S^2\setminus U^2}),\overline{\omega}_{U^2},A)\\
			&=(K^2)_{S^2}^{\do(U^2,\Q^2,\L^2)}(f(\omega),A).
		\end{alignat*}
		This ends the proof.
	\end{proof}
	\subsection{Proofs for Section~\ref{subsec:causal_effect}}\label{app:proofs3}
	\nocausaleffect*
	\begin{proof}[Proof of Lemma~\ref{lem:no_causal_effect}]
		Consider \(A\in \mc{H}^2_{V^2}\) and any \(S^2\subset T^2\). Then for any \(\omega'\in\Omega^2\) we find an \(\omega\in \Omega^1\) such that \(f(\omega)=\omega'\). Using interventional consistency and \(f^{-1}(A)\in \mc{H}^1_{V^1}\) we conclude 
		\begin{alignat*}{2}
			K^2_{S^2}(\omega',A)&=K^1_{\rho^{-1}(S^2)}(\omega,f^{-1}(A))\\
			&=K^1_{\rho^{-1}(S^2)\setminus\rho^{-1}(U^2)}(\omega,f^{-1}(A))\\
			&=K^1_{\rho^{-1}(S^2\setminus U^2)}(\omega,f^{-1}(A))\\
			&=K^2_{S^2\setminus U^2}(\omega',A).
		\end{alignat*}
		This ends the proof.
	\end{proof}
	\activecausaleffect*
	\begin{proof}[Proof of Lemma~\ref{lem:active_effect}]
		Since \(\mc{H}^2_{U^2}\) has an active causal effect on \(\mc{H}^2_{V^2}\) in \(\mc{C}^2\), we find that there is an \(\omega'\in \Omega^2\) and an \(A\in \mc{H}^2_{V^1}\) such that
		\[K_{U^2}^2(\omega',A)\neq\P^2(A).\]
		By surjectivity there is \(\omega\in\Omega^1\) such that \(\omega'=f(\omega)\) and thus
		\begin{alignat*}{2}
			K^1_{U^1}(\omega,f^{-1}(A))&=K_{U^2}^2(\omega',A)\\
			&\neq\P^2(A)\\
			&=\P^1(f^{-1}(A)).
		\end{alignat*}
		The claim follows because \(f^{-1}(A)\in \mc{H}^1_{U^1}\).
	\end{proof}
	\sources*
	\begin{proof}[Proof of Theorem~\ref{lem:sources}]
		Our goal is to show that \(K_{U^2}^2(\cdot,A)\) is a version of the conditional probability \(\P^2_{\mathscr{H}^2_{U^2}}(A)\) for \(A\in\mc{H}^2_{V^2}\). It is sufficient to show that for all \(B\in\mc{H}_{U^2}^2\) the following relation holds
		\begin{equation}\label{eq:rel_to_prove}
			\int \P^2(\d \omega') \1_A(\omega')\1_B(\omega')
			= \int \P^2(\d \omega')\1_B(\omega') K_{U^2}^2(\omega', A).
		\end{equation}
		Using that \((f,\rho)\) is a perfect abstraction, \(f^{-1}(A)\in\mc{H}^1_{V^1}\), \(f^{-1}(B)\in\mc{H}^1_{U^1}\), and that \(\mc{H}^1_{U^1}\) is a local source of \(\mc{H}^1_{V^1}\) we find
		\begin{alignat*}{2}
			\int\P^2(\d\omega')\1_A(\omega')\1_B(\omega')&=\int f_\ast\P^1(\d\omega')\1_A(\omega')\1_B(\omega')\\
			&=\int\P^1(\d\omega)\1_A(f(\omega))\1_B(f(\omega))\\
			&=\int\P^1(\d\omega)\1_{f^{-1}(A)}(\omega)\1_{f^{-1}(B)}(\omega)\\
			&=\int\P^1(\d\omega)K^1_{U^1}(\omega,f^{-1}(A)) \1_{f^{-1}(B)}(\omega)\\
			&=\int\P^1(\d\omega)K^2_{U^2}(f(\omega),A)\1_{B}(f(\omega))\\
			&=\int f_\ast\P^1(\d\omega')K^2_{U^2}(\omega',A)\1_{B}(\omega')\\
			&=\int\P^2(\d\omega')K^2_{U^2}(\omega',A)\1_{B}(\omega').
		\end{alignat*}
		Thus we have shown that \eqref{eq:rel_to_prove} holds and the proof is completed.
	\end{proof}
	
	\section{Equivalence of causal mechanism definitions}
	\label{app:factorization}
	Let us here briefly comment on the notation and equivalence of the notions $K_S(\omega,A)$ to $K_S(\omega_S,A)$. 
	To clarify the equivalence let us recall the factorization lemma.
	\begin{lemma}[Factorization Lemma]
		Let $T:\Omega\to \Omega'$ be a map and $(\Omega', \mc{A}')$
		a measurable space. A function $f:\Omega\to [0,1]$
		is $\sigma(T)/\mc{B}([0,1])$ measurable if and only if $f=g\circ T$
		for some $\mc{A}'/\mc{B}([0,1])$ measurable function 
		$g:\Omega'\to[0,1]$.
	\end{lemma}
	We apply this where for $T$ we use the projection $\pi_S:(\Omega,\mc{H})\to (\Omega_S,\mc{H}_S)$.
	Then, by the factorization lemma, a $\mc{H}_S$ measurable map $K_S(\cdot, A):\Omega\to [0,1]$
	gives rise to a map $K_S'(\cdot, A):\Omega_S\to [0,1]$
	such that 
	\begin{align}\label{eq:fact_lemma_applied}
		K_S(\omega,A)=K_S'(\pi_S\omega,A)=K_S'(\omega_S,A).
	\end{align}
	Moreover, this maps is unique because $\pi_S$ is surjective.
	On the other hand, give $K_S'(\cdot, A):\Omega_S\to [0,1]$ we can just define $K_S(\cdot,A):\Omega\to [0,1]$
	by \eqref{eq:fact_lemma_applied}. For this reason $K_S(\omega,A)$ and
	$K_S(\omega_S,A)$ can be used equivalently when identified as explained here  and with a slight abuse of notation we will use both depending on the context without indicating the different domains.
	
\end{document}